\documentclass{amsart}
\usepackage{amsmath,amsthm,amsfonts,amssymb,,amscd, epsfig, mathrsfs, xypic, setspace}
\usepackage[vcentermath,enableskew]{youngtab}
\usepackage{color}
\usepackage{microtype}
\usepackage[utf8]{inputenc}
\usepackage{pst-node, pstricks}
\usepackage{url}

\usepackage{ascmac}
 \usepackage{fancybox}
 
\xyoption{all}

\newtheorem{thm}{Theorem}[section]
\newtheorem*{thm*}{Theorem}

\newtheorem*{lem*}{Lemma}
\newtheorem{cor}[thm]{Corollary}
\newtheorem{conj}{Conjecture}[section]
\theoremstyle{definition} 
\theoremstyle{definition} 
\theoremstyle{definition} \newtheorem{defn}[thm]{Definition}
\theoremstyle{definition} \newtheorem{rem}[thm]{Remark}
\theoremstyle{definition} 
\theoremstyle{definition} \newtheorem{ex}[thm]{Example}
\newtheorem{PAP}{Pattern Avoidance Property}
\newtheorem*{finalfirstPAP}{Pattern Avoidance Property 10-1}
\newtheorem*{finalsecondPAP}{Pattern Avoidance Property 10-2}
\newtheorem*{finalthirdPAP}{Pattern Avoidance Property 10-3}
\newtheorem{ques}{Question}

\DeclareMathOperator{\fl}{fl}

\newcommand{\Z}{\mathbb{Z}}

\newcommand{\ci}{\bullet}
\newcommand{\flags}{\mathcal{F}l_n}
\newcommand{\enci}{\put(3,3){\circle{10}}}
\newcommand{\given}{\,\, | \,\,}

\newcommand{\glB}{\tilde{B}}

\newcommand{\gn}{ \bullet }
\newcommand{\abechs}[2]{ \binom{#1}{#2} }
\newcommand{\mfS}{ S }
\newcommand{\sectionspace}{\vspace{0.2in}}

\psset{unit=1pt,arrowsize=4pt,linewidth=.5pt,linecolor=black}

\definecolor{541}{rgb}{1,0.8,0.2}
\definecolor{530}{rgb}{1,0.6,0.0}
\definecolor{light-blue}{rgb}{0.6,0.65,1}
\definecolor{204}{rgb}{0.4,0.0,0.8}
\definecolor{purple}{rgb}{0.4,0.0,0.8}
\definecolor{035}{rgb}{0,0.6,1}
\definecolor{005}{rgb}{0,0.0,1}
\definecolor{red}{rgb}{1,0,0}

\begin{document}
\title[Consequences of the Lakshmibai-Sandhya Theorem]{Consequences of the Lakshmibai-Sandhya Theorem: \\ the ubiquity of permutation patterns \\ in Schubert calculus and related
  geometry}

\author[H. Abe]{Hiraku Abe} %
\address{Department of Mathematics, Faculty of Science and Engineering, Waseda University, 3-4-1 Okubo, Shinjuku, Tokyo 169-8555 JAPAN}
\email{hirakuabe@globe.ocn.ne.jp} 

\author[S. Billey]{Sara Billey} %
\address{Department of Mathematics, Padelford C-445, University of
  Washington, Box 354350, Seattle, WA 98195-4350 }
\email{billey@math.washington.edu}
\thanks{The second author was partially supported by grant DMS-1101017
  from the NSF}  

\date{\today}

\begin{abstract}
  In 1990, Lakshmibai and Sandhya published a characterization of
  singular Schubert varieties in flag manifolds using the notion of
  pattern avoidance.  This was the first time pattern avoidance was
  used to characterize geometrical properties of Schubert varieties.
  Their results are very closely related to work of Haiman, Ryan and
  Wolper, but Lakshmibai-Sandhya were the first to use that language
  exactly.  Pattern avoidance in permutations was used historically by
  Knuth, Pratt, Tarjan, and others in the 1960's and 1970's to
  characterize sorting algorithms in computer science.  Lascoux and
  Sch\"utzenberger also used pattern avoidance to characterize
  vexillary permutations in the 1980's.  Now, there are many
  geometrical properties of Schubert varieties that use pattern
  avoidance as a method for characterization including Gorenstein,
  factorial, local complete intersections, and properties of
  Kazhdan-Lusztig polynomials.  These are what we call
  consequences of the Lakshmibai-Sandhya theorem.  We survey the many
  beautiful results, generalizations, and remaining open problems in
  this area.  We highlight the advantages of using pattern avoidance
  characterizations in terms of linear time algorithms and the ease
  of access to the literature via Tenner's Database of Permutation
  Pattern Avoidance.   This survey is based on lectures by the second author at
  Osaka, Japan 2012 for the Summer School of the Mathematical Society
  of Japan based on the topic of Schubert calculus.
\end{abstract}

\maketitle

\sectionspace
\section{Introduction}

Modern Schubert calculus is the study of effective methods to compute
the expansion coefficients for the cup product of cohomology classes of Schubert
varieties:
\[
[X_u] \cdot [X_v]  = \sum  c_{u,v}^{w} [X_w].
\]
These coefficients $c_{u,v}^{w}$ are called structure constants with
respect to the Schubert classes $[X_w]$, and it is known that the
structure constants are non-negative integers. In fact, each
$c_{u,v}^{w}$ is the \textit{intersection number} of three Schubert
varieties $X_u, X_v$ and $X_{w_{0} w}$; they count the number of points
of the intersection of those three varieties placed in generic
positions. Observe that this is both a combinatorial and a geometrical
statement.

For Schubert varieties in Grassmannians, we already have many tools
for computing the structure constants for the cup product:
\textit{Littlewood-Richardson tableaux}, \textit{Yamanouchi words},
\textit{Knutson-Tao puzzles}, \textit{Vakil's toric degenerations}.
In general, we have not found analogs of all these beautiful tools for other types of Schubert varieties. We need to understand both the combinatorics and geometry of Schubert varieties in order to do Schubert calculus for all types of Schubert varieties.

In this article, we will focus on the combinatorics and geometry
related to the tangent spaces of Schubert varieties and
characterizations of smoothness and rational smoothness.  The
mathematical tools we will use also arise in Schubert calculus, but we
will not make the connections explicit.  For the record, the most
explicit connection between characterizations of smoothness and
Schubert calculus come from Kumar's criterion and the Kostant
polynomials.  See \cite{BLak,kumar,tymoczko.2013} for more details.  

We begin with a review of Schubert varieties in flag manifolds.  Then
we will present the celebrated Lakshmibai-Sandhya Theorem
characterizing smooth Schubert varieties using permutation pattern
avoidance.  We will give a total of 10 properties of Schubert
varieties in flag manifolds that are completely characterized by
pattern avoidance or a variation on that theme.  We describe a method
for extending permutation pattern avoidance to all Coxeter groups and
discuss some geometrical properties characterized by Coxeter pattern
avoidance more generally.  We give pointers to some useful
computational tools for studying Schubert geometry and beyond.
Finally, we present many open problems in this area.

We want to highlight the fact that there are computational advantages
of using permutation patterns to characterize interesting properties
such as smoothness of Schubert varieties.  Naively, avoiding a finite
set of patterns of length at most $k$ leads to a polynomial time
algorithm of $O(n^k)$ by brute force testing of all $k$-subsets.  As
$k$ and $n$ get large, such algorithm is intractable.  In fact,
deciding if one permutation is contained in another is an NP-complete
problem \cite{Bose.Buss.Lubiw}.  Remarkably, Guillemot and Marx
\cite{Guillemot.Marx} recently showed that for every permutation $v
\in S_k$ there exists an algorithm to test if $w \in S_n$ contains $v$
which runs in \textit{linear time}, $O(n)$!  This is a major
improvement over brute force verification.  It is often far from
obvious that an $O(n)$ time algorithm exists for the geometric or
algebraic properties characterized by pattern avoidance in this paper.

Another major advantage of permutation pattern characterizations is
that they provide efficient fingerprints for theorems
\cite{billey.tenner}.  Tenner's Database of Permutation Pattern
Avoidance (DPPA) provides a growing collection of known properties
characterized by patterns with references to the literature
\cite{dppa}.  This allows researchers to connect new theorems and
conjectures with known results in a format free of language or
notational differences.

\sectionspace
\section{Preliminaries}\label{prelim}
\subsection{The Flag Manifold}\label{sub:flags}
\begin{defn}
A \textit{complete flag} $F_{\ci}= (F_{1},\dots , F_{n})$ in $\mathbb{C}^{n}$ is a nested sequence of vector spaces such that $\mathrm{dim} (F_{i})=i$ for $1\leq i \leq n$.  A flag $F_\ci$ is determined by an ordered basis $\langle f_{1},f_{2},\dots, f_{n } \rangle $ where $F_{i} = \mathrm{span}\langle f_{1},\dots , f_{i} \rangle$.
\end{defn}

Let $e_1,e_2,\ldots, e_n$ be the standard basis for $\mathbb{C}^{n}$.
The \textit{base flag} is $E_{\ci}= (E_1,E_2,\ldots, E_n)$ where $E_i =
\langle e_{1},e_{2}, \ldots , e_{i} \rangle$.  Let $F_\ci $ be any
flag given by the ordered basis $ \langle f_{1},f_{2},\dots ,f_{n }
\rangle$.  Writing each basis element $f_i$ as a column vector in
terms of the $e_i$'s, we obtain an $n\times n$-non-singular matrix
whose column vectors are the basis $f_1,\cdots,f_n$. In this
presentation, we can multiply the matrix by a non-zero scalar or we
can add the $i$-th column to the $j$-th column where $i<j$ and it
still represents the same flag. So, a flag can always be presented by
a matrix in \textit{canonical form}; the lowest non-zero entry of each
column is 1, and the entries to its right are all zeros.
\begin{ex}The following two matrices represent the same flag $F_{\ci} =   \langle 2 e_{1}+ e_{2},\hspace{.1in} 2 e_{1} +e_{3},\hspace{.1in} 7e_{1}+e_{4},\hspace{.1in} e_{1}  \rangle$:
\begin{align*}
\left[\begin{array}{cccc}
6 &	4 &	 9 &	 0 \\
3 &	0 &	0 &	1\\
0 &    2 &     1 &	0\\
0 &    0 &      1 &      0
\end{array} \right]
  \sim
\left[\begin{array}{cccc}
2 &	2 &	 7 &	 1 \\
1 &	0 &	0 &	0\\
0 &   1 &     0 &	0\\
0 & 0 &	 1 &	0
\end{array} \right].
\end{align*}
The right hand side is the canonical form.
\end{ex}

It also follows that two non-singular matrices represent the same flag
if and only if one is the other multiplied by an upper triangular
matrix.  That is, we have an identification $\flags(\mathbb{C}) =
GL_{n}(\mathbb{C}) / \glB$ where $\glB\subset GL_{n}(\mathbb{C})$ is
the set of invertible upper triangular matrices.  Similarly, we can
rescale any invertible matrix by the inverse of its determinant and
get another matrix representing the same flag.  Hence, letting $B$ be
the set of upper triangular matrices in $SL_n(\mathbb{C})$, we see
that
\begin{align*}
\flags(\mathbb{C}) =  GL_{n}(\mathbb{C}) / \glB = SL_{n}(\mathbb{C}) / B.
\end{align*}

\subsection{Flags and Permutations}
If a flag is written in canonical form, the leading 1's form a
permutation matrix. This matrix is called \textit{the position} of the
flag $F_{\ci}$ with respect to the base flag $E_{\ci} $, and is denoted by
$\text{position}(E_{\ci},F_{\ci})$.
\begin{ex} 
\begin{align*}
F_{\ci} =   \langle 2 e_{1}+ e_{2},\hspace{.1in} 2 e_{1} +e_{3},\hspace{.1in} 7e_{1}+e_{4},\hspace{.1in} e_{1}  \rangle \approx
\left[\begin{array}{cccc}
2 &	2 &	 7 &	 \enci 1 \\
\enci 1&0 &	0 &	0\\
0 &     \enci 1 &     0 &	0\\
0 &     0 &	 \enci 1 &	0
\end{array} \right]
\end{align*}
\end{ex}

Note that there are many ways to represent a permutation; as a
bijection from $[n]:=\{1,2,\ldots,n\}$ to itself, matrix notation,
two-line notation, one-line notation, rank table, diagram, string
diagram, reduced word etc.  Each of these representations is useful in
some way or another for the study of Schubert varieties so we advise
the reader to become comfortable with all of them simultaneously and
choose the right one for the proof at hand.  Note, we have not found
much use for cycle notation for permutations in this context so we
will not ever use that notation here.   

To be precise, we use the following notation: for a permutation
$w\colon [n]\rightarrow[n]$ in the symmetric group $S_n$, we denote by the
same symbol $w=w_1w_2\ldots w_n$ the permutation matrix which has 1's
in the $(w_j,j)$-th entries for $1\leq j\leq n$ and 0's
elsewhere. Permutation multiplication is consistent with matrix
multiplication using this notation.  In particular, if $t_{ij}$ is the
transposition interchanging $i$ and $j$, then the one-line notation for
$wt_{ij}$ agrees with $w$ in all positions except $i$ and $j$ where
the entries are switched.  The permutation $t_{ij}w$ has the values
$i$ and $j$ switched.

The \textit{rank table} $rk(w)$ is obtained from the matrix $w$ by
setting 
$$
rk(w)[i,j] = \# \{h \in [j] \colon w(h) \in [i]\},
$$
 i.e.  the rank of the submatrix of $w$ with lower
right corner $[i,j]$ and upper left corner $[1,1]$.  

A \textit{string diagram} of a permutation for $w$ is a braid with the
strings proceeding from the initial ordering to the permuted order
given by $w=w_1w_2\ldots w_n$ in such a way that no three strings
cross at any point.  A wiring diagram is a string diagram with exactly
one crossing on each row.  A wiring diagram in which no two strings
cross twice is said to be \textit{reduced}.  Starting at the top of a
reduced wiring diagram, one can read off the index of the first string
in each crossing to obtain a corresponding \textit{reduced word}.  All
reduced words for $w$ have the same length, denoted $\ell(w)$.
Furthermore, the length of $w$ is the number of \textit{inversions}
for $w$, $\ell(w)=\# \{w(i)>w(j)\colon i<j \}$.    

The \textit{diagram of a permutation} $w$ is obtained from the matrix
of $w^{-1}$ by removing all cells in an $n\times n$ array which are
weakly to the right or below a 1 in $w^{-1}$.  The remaining cells
form the diagram $D(w)$.  The cells of $D(w)$ are in bijection with
the inversions of $w$.  One can recover $w$ either from its diagram or
its inversion set.  It is unfortunate that the diagram is defined in
terms of $w^{-1}$, but that is the most common convention in the
literature \cite{M2}.

\vspace{0.1in}

\begin{ex}
\[
\left[\begin{array}{cccc}
0 &	 0 &	 0 &	 1 \\
1 &	 0 &	 0 &	0\\
0 &      1 &     0 &	 0\\
0 &      0 &	1  &	0
\end{array} \right] = \left[\begin{array}{cccc}
1 &	 2 &	 3 &	 4\\
2 &	 3 &	 4 &	 1
\end{array} \right]
= 2 3 4 1 = 
\left[\begin{array}{cccc}
0 &	0 &	0 &	 1 \\
1 &    1 &	 1 &	2\\
1 &    2 &     2 &	 3\\
1 &    2 &	 3 &	4
\end{array} \right] 
\]
\[
\begin{array}{c}
\text{matrix}\\
\text{notation}
\end{array}
\hspace{.5in}
\begin{array}{c}
\text{two-line}\\
\text{notation}
\end{array}
\hspace{.2in}
\begin{array}{c}
\text{one-line}\\
\text{notation}
\end{array}
\hspace{.3in}
\begin{array}{c}
\text{rank}\\
\text{table}
\end{array}
\]
\hspace{.88in}
$\displaystyle 
\begin{array}{cccc}
* &	 . &	 . &	 . \\
* &	.&	 . &	.\\
* &    . &     . &	 .\\
 . & . & . &	.
\end{array}  
\hspace{.15in}
= \hspace{.85in} = \hspace{.2in}
(1,2,3)
$
\vspace{0.26in}
\begin{figure}[h]
\vspace{-1.0in}
\hspace{-.7in}
\includegraphics[height=1.6cm]{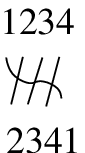}
\end{figure}

\vspace{-0.09in}
\[
\begin{array}{c}
\text{diagram of a }\\
\text{permutation}
\end{array}
\hspace{.1in}
\begin{array}{c}
\text{string diagram}
\end{array}
\hspace{.1in}
\begin{array}{c}
\text{reduced}\\
\text{word}
\end{array}
\hspace{1.0in}
\end{equation*}
\end{ex}
\vspace{.1in}

\subsection{Schubert Cells and Schubert Varieties in $\flags(\mathbb{C})$}
For a permutation $w \in S_n$, the Schubert cell
$C_{w}(E_{\ci})\subset \flags(\mathbb{C})$ is the set of all flags
$F_{\ci}$ with $\mathrm{position}(E_{\ci},F_{\ci})=w$. Equivalently,
we can write $C_{w}(E_{\ci})$ as
\begin{align*}
C_{w}(E_{\ci}) = 
\{F_{\ci} \in \flags(\mathbb{C})
\given \mathrm{dim}(E_{i} \cap F_{j}) = \mathrm{rk}(w)[i,j] \text{ for all } 1\leq i,j \leq n \}.
\end{align*}
Note, the flag $w_{\ci}$ represented by the permutation matrix for $w$
is in $C_w$ by the rank conditions.

\begin{ex}
\begin{align*}
F_{\ci} =    
\left[\begin{array}{cccc}
2 &	2 &	 7 &	 \enci 1 \\
\enci 1&0 &	0 &	0\\
0 &     \enci 1 &     0 &	0\\
0 &     0 &	 \enci 1 &	0
\end{array} \right]
\in 
\displaystyle C_{2341} = 
\left\{\left[ 
\begin{array}{cccc}
x &	 y &	 z &	 1 \\
1 &	.&	 . &	.\\
. &    1 &     . &	 .\\
 . & . & 1 &	.
\end{array}
 \right]:  x,y,z\in \mathbb{C} \right\} 
\end{align*}
\end{ex}

\bigskip

It is easy to observe the following properties for each permutation $w$.
\begin{itemize}
\item[(i)] The dimension of a Schubert cell is $  \mathtt{dim_{\mathbb{C}}}(C_{w})=  \ell(w)$.
\item[(ii)] The indeterminates for the canonical matrices in $C_w$ all
  lie in the entries of the diagram $D(w^{-1})$.  
\item[(iii)] $C_{w} = \glB\cdot w_\bullet $  is a $\glB$-orbit using the left $\glB$
  action on flags given by multiplication of matrices.  See Example~\ref{ex:bs}.  
\end{itemize}


\begin{ex}\label{ex:bs}
For arbitrary $b_{i,j}$'s with $b_{1,1},b_{2,2},b_{3,3}, b_{4,4}$ non-zero, we have
\begin{small}
\[
 \left[ \begin {array}{cccc} b_{{1,1}}&b_{{1,2}}&b_{{1,3}}&b_{{1,4}}
\\ \noalign{\medskip}0&b_{{2,2}}&b_{{2,3}}&b_{{2,4}}
\\ \noalign{\medskip}0&0&b_{{3,3}}&b_{{3,4}}\\ \noalign{\medskip}0&0&0
&b_{{4,4}}\end {array} \right] 
 \left[ \begin {array}{cccc} 0&0&0&1\\ \noalign{\medskip}1&0&0&0
\\ \noalign{\medskip}0&1&0&0\\ \noalign{\medskip}0&0&1&0\end {array}
 \right]
=
\left[ \begin {array}{cccc} b_{{1,2}}&b_{{1,3}}&b_{{1,4}}&b_{{1,1}}
\\ \noalign{\medskip}b_{{2,2}}&b_{{2,3}}&b_{{2,4}}&0
\\ \noalign{\medskip}0&b_{{3,3}}&b_{{3,4}}&0\\ \noalign{\medskip}0&0&b
_{{4,4}}&0\end {array} \right] \in C_{2341}.
\]
\end{small}
\end{ex}

\begin{defn}
The Schubert variety $X_{w}(E_{\ci})$ of a permutation $w$ is defined to be the closure of $C_{w}(E_{\ci})$ under the Zariski topology. As in the case for Schubert cells, $X_{w}(E_{\ci})$ can be written by the rank conditions:
\begin{align*}
X_{w}(E_{\ci}) = \{F_{\ci} \in \flags
\given \mathrm{dim}(E_{i} \cap F_{j}) \geq \mathrm{rk}(w)[i,j] \text{ for all } 1\leq i,j \leq n \}.
\end{align*}
\end{defn}

\begin{ex}
\begin{align*}
\left[\begin{array}{cccc}
\enci 1 &	 0 &	 0 &	 0 \\
0 &	* &	 * &	\enci 1\\
0 &    \enci 1 &     0 &	 0\\
0 &	0 & \enci  1 &	 0 
\end{array} \right] 
\in 
\displaystyle X_{2341}(E_{\ci}) = 
\overline{
\left\{
\left[\begin{array}{cccc}
* &	 * &	 * &	 1 \\
1 &	 0 &	 0 &	0\\
0 &      1 &     0 &	 0\\
0 &      0 &	 1 &	0
\end{array} \right]
 \right\}
}
\end{align*}
\end{ex}

\subsection{Combinatorics and Geometry}
Since Schubert cells are $\glB$-orbits, Schubert varieties are
$\glB$-invariant by their definition. So each Schubert variety is
equal to a disjoint union of Schubert cells 
\begin{align}\label{prelim 100}
  X_{w} = \bigcup_{v\leq w} C_{v}.
\end{align}
Thus, the containment relation on Schubert varieties $X_v \subset X_w$
defines a partial order on permutations $v\leq w$. This partial order
has a nice description: for a permutation $w$ and integers $1\leq
i<j\leq n$, we say $w < w t_{ij}$ if $w(i)<w(j)$.  \textit{Bruhat
  order} (discovered by Ehresmann 1934 \cite{ehresmann.1934}, see also
Chevalley 1958 \cite{chevalley.1958}) is defined to be the transitive closure of this
relation.

\begin{ex}
The following is the Hasse diagram of the Bruhat order on permutations in $S_{3}$.
\vspace{-0.1in}
\setlength{\unitlength}{1cm}
\newcommand{\perm}[1]{{\raisebox{-.1\unitlength}{\makebox(0,0)[b]{${#1}$}}}}
\begin{center}
\raisebox{1ex}{\begin{picture}(0,2.3)
\put(1.5,0){\perm{132}}\put(1.5,1){\perm{231}}
\put(.5,-1){\perm{123}}\put(.5,1.8){\perm{321}}
\put(-.6,0){\perm{213}}\put(-.6,1){\perm{312}}
\put(0,0){\line(0,1){1}}
\put(1,0){\line(0,1){1}}
\put(0,0){\line(1,1){1}}
\put(.5,1.5){\line(1,-1){.5}}
\put(.5,1.5){\line(-1,-1){.5}}
\put(.5,-.5){\line(-1,1){.5}}
\put(.5,-.5){\line(1,1){.5}}
\put(0,1){\line(1,-1){1}}
\end{picture}}
\end{center}
\vspace{.4in}
The Hasse diagram of $S_{n}$ is self dual, rank symmetric and rank unimodal.
\end{ex}

\begin{ex}
The Hasse diagram of $S_{4}$ is drawn in Figure \ref{Hasse of S4}.
\begin{figure}[h]
\includegraphics[height=14cm]{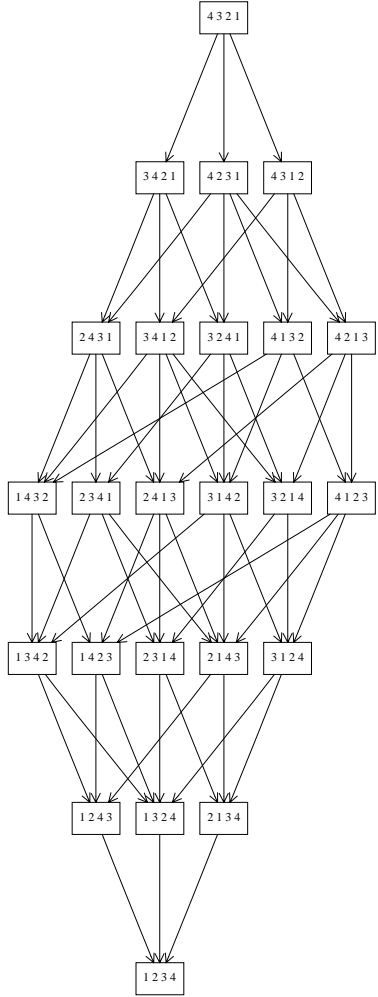}
\caption{The Hasse diagram of $S_{4}$}\label{Hasse of S4}
\end{figure}
\end{ex}

One of the benefits of Bruhat order is a description of the \textit{Poincar\'e polynomials} of Schubert varieties.
More precisely, the Poincar\'e polynomial for $H^{*}(X_{w})$ is given by
\begin{align*}
P_{w}(t)= \sum_{v \leq w} t^{2l(v)}.
\end{align*}
Because only even exponents appear in the Poincar\'e polynomials
above, we often abuse notation and define 
\begin{align*}
P_{w}(t)= \sum_{v \leq w} t^{l(v)}.
\end{align*}

\begin{ex}
  For $w=3412$, the following permutations are in the interval below
  $3412$ in Bruhat order.  
$$
\begin{matrix}
4: & (3 4 1 2)\\
&\\
3: & (3 1 4 2) (3 2 1 4) (1 4 3 2) (2 4 1 3)\\
&\\
2: & (3 1 2 4) (1 3 4 2) (2 1 4 3) (2 3 1 4) (1 4 2 3)\\
&\\
1: & (2 1 3 4) (1 2 4 3) (1 3 2 4)\\
&\\
0: & (1 2 3 4)
\end{matrix}
$$
\end{ex}
So $P_{3412}(t) = 1 + 3t + 5 t^{2} + 4 t^{3} +t^{4}$. One can see that
the Schubert variety $X_{3412}$ is not smooth since its Poincar\'e
polynomial is not symmetric (palindromic) which implies that 
Poincar\'e duality does not hold for $H^*(X_{3412})$.
\\

There are several interesting things about Bruhat order. We will encounter some of them in the rest of the paper. We will focus on the relationship between singularities of Schubert varieties and pattern avoidance of permutations. 

We leave to the reader the following exercises.
\begin{enumerate}
\item The boundary of $X_{w}$ has irreducible components given by the
  Schubert varieties $X_v$ such that $v <w$ such that
  $\ell(v)=\ell(w)-1$.  
\item $C_{w}$ is a dense open set in $X_{w}$.
\item $X_{w}$ embeds into a product of projective spaces via Pl\"ucker
  coordinates.  A matrix is mapped under this embedding to the list
  all of its lower left minors in a given order.
\item If $w_{0}=[n,n-1,\dots,1]$, then $GL_{n}/\glB=X_{w_{0}}$. 
\item The point $w_{0}$ has an affine neighborhood $C_{w_{0}}$ of
dimension $\binom{n}{2}$ and a local coordinate system.  A generic point
$g$ has an affine neighborhood $g w_{0}C_{w_{0}}$ in $\flags$. 
\item $GL_{n}$ acts transitively on the points in the flag manifold so
it is a manifold and a projective variety.  
\item The flag manifold is smooth (i.e. non-singular at every
point).
\end{enumerate}

\sectionspace
\section{Smooth Schubert varieties}\label{geom of Schubert}

Say we wish to determine which Schubert varieties are smooth and which
are not.  There are several combinatorial and geometrical observations
which makes this determination easier to characterize than a typical
variety.

First, an affine variety is smooth at a point if the dimension of its
tangent space equals the dimension of the variety near that point.  If
the variety is given in terms of the vanishing of certain polynomials,
then one can check the dimension of the tangent space by computing the
rank of the Jacobian matrix for those polynomials evaluated at the
point.  The rank is smaller than expected if and only if all minors of
a certain size vanish.  Thus, the set of points where the variety is
not smooth is itself a variety called the \textit{singular locus}.

A priori, to determine if a variety is smooth at every point, one must 
check the dimension of the tangent space at every point.  For Schubert
varieties, we make an easy observation. A point $p\in C_{v}\subset
X_{w}$ is singular in $X_{w}$ if and only if every point in $C_{v}$ is
singular in $X_{w}$ since the Schubert cell $C_{v}$ is a
$\glB$-orbit. Recalling that the singular locus of a variety is a closed
set, the equality \eqref{prelim 100} implies that each Schubert
variety $X_{w}$ is smooth if and only if $X_{w}$ is smooth at the
identity matrix $I$. One can check the singularity at the identity by
writing down the defining equations of $X_{w}$ around an affine
neighborhood of $X_{w}$ around $I$ (for example, $X_{w} \cap w_{0}
C_{w_{0}}$) and check the rank of the Jacobian matrix of the defining
polynomials. However, there is another way which provides a more
unified tool for the study of the singularity of Schubert varieties
using Lie algebras.

\subsection{Lie algebras and tangent spaces of Schubert varieties}
Recall from Section~\ref{sub:flags} that the flag variety can be
identified with the quotient of a semisimple algebraic group:
\begin{align*}
\flags = GL_{n}(\mathbb{C})/\glB = SL_{n}(\mathbb{C}) / B
\end{align*}
where $\glB$ is the set of upper triangular matrices in
$GL_{n}(\mathbb{C})$ and $B=\glB\cap SL_{n}(\mathbb{C})$.  The tangent
space of $SL_n$ is isomorphic as a vector space to its Lie algebra,
which is known to be the $n \times n$ trace zero matrices over
$\mathbb{C}$.  The Lie algebra of $B$ is the subalgebra of upper
triangular matrices with trace zero.  Let $G=SL_{n}(\mathbb{C})$,\
$\mathfrak{g} =\mathrm{Lie}(SL_{n})$ and $\mathfrak{b} =\mathrm{Lie}(B
)$. Then the tangent space of $G/B$ at the identity matrix is
isomorphic to $\mathfrak{g}/\mathfrak{b}$. Denoting by $E_{i,j}$ the
matrix with 1 in the $(i,j)$-entry and 0's elsewhere, we obtain a basis
for $\mathfrak{g}/\mathfrak{b}$ by  
\[
\mathfrak{g}/\mathfrak{b} = \mathrm{span}\{E_{j,i}\colon 1\leq i<j \leq n \}.
\]
Observe that there is a natural bijection between the basis elements
$\{E_{j,i}\colon i<j \}$ and $R:=\{t_{i,j}\colon i<j \}$ the set of reflections.

More generally, for any $v \in S_n$, the tangent space to $G/B$ at $v$ is given by
\begin{align}\label{geom of Schubert 100}
T_{v}(G/B) = v \left(\mathfrak{g}/\mathfrak{b} \right)  v^{-1}
=\mathrm{span}\{ E_{v(j),v(i)} \colon i<j \} . 
\end{align}
Why? Because, $ G/v B v^{-1}$, is an isomorphic copy of the flag
manifold $G/B$ but with respect to the base flag $v_\bullet$. Here the
flag $v_\bullet= vB$ is fixed by the left action of $vBv^{-1}$.

It is an easy exercise to check 
\begin{align*}
&v \ E_{ij}\ v^{-1}= E_{v(i),v(j)}, \\
&t_{v(i),v(j)}\ v = v\ t_{ij} 
\end{align*}
for any $1\leq i<j\leq n$. 
The next theorem gives us an explicit description of a basis of the tangent space of each Schubert variety. 
\begin{thm}\label{geom of Schubert 150}
\emph{(Lakshmibai-Seshadri \cite{Lak-Sesh.1984})}
For $v\leq w \in S_n$, the tangent space of $X_{w}$ at $v$ is given by
\begin{align*}
T_{v}(X_{w}) & \cong \mathrm{span}\{E_{v(j),v(i)} \colon i<j,\ \ vt_{ij} \leq w \}, 
\end{align*}
and hence we obtain
\begin{align*}
\mathrm{dim} \ T_{v}(X_{w}) &= \# \{(i<j)\colon \  vt_{ij} \leq w \}.
\end{align*}
\end{thm}
\begin{proof} Recall from the definition of a Lie algebra that
$E_{v(j),v(i)} \in T_{v}(X_{w})$ is equivalent to $(I + \varepsilon
E_{v(j),v(i)} )v \in X_{w}$ for infinitesimal $\varepsilon>0$ where
we can assume $\varepsilon^2 =0$.  Think of $(I + \varepsilon
E_{v(j),v(i)} )$ as a matrix in $G$ acting on the left of the flag
$v_\bullet$ by moving the flag a little bit in the direction of
$E_{v(j),v(i)}$.  In particular, $(I + \varepsilon
E_{v(j),v(i)} )v = v + \varepsilon E_{v(j),v(i)}v =$ $v + \varepsilon E_{v(j),i} \in X_{v}$ if and
only if $v(i)>v(j)$ which is equivalent to $vt_{ij}\leq v$.  Since $v
\leq w$ implies $T_{v}(X_{v})\subset T_{v}(X_{w})$ we see that
$E_{v(j),v(i)}$ is in $T_{v}(X_{w})$ whenever $v(i) > v(j)$.

On the other hand, if $v(i)<v(j)$ then $v + \varepsilon E_{v(j),i} \in
C_{vt_{ij}}$ and so $E_{v(j),v(i)} \in T_{v}(X_{w}) $ if and only if
$vt_{ij} \leq w$.  Thus, in either case $E_{v(j),v(i)} \in
T_{v}(X_{w}) $ if and only if $vt_{ij} \leq w$.   Thus, $\mathrm{dim} \ T_{v}(X_{w}) \geq \# \{(i<j)\colon \  vt_{ij} \leq w \}.$

To prove $\mathrm{dim} \ T_{v}(X_{w}) \leq \# \{(i<j)\colon \ vt_{ij}
\leq w \}$, assume there exists coefficients $a_{i,j}$ for $1\leq
i<j\leq n$ such that $v + \varepsilon \sum a_{i,j} E_{v(j),i} \in
X_{w}$.  Say $v + \varepsilon \sum a_{i,j} E_{v(j),i} \in C_{v'}$ for
some $v'\leq w $.  Since $\varepsilon<<1$, none of the minors in $v$
which are nonzero will vanish in $v + \varepsilon \sum a_{i,j}
E_{v(j),i} $, so the rank table for $v + \varepsilon \sum a_{i,j}
E_{v(j),i} $ dominates the rank table for $v$ in every position.
Hence, $v\leq v'\leq w$.  Thus, for each $a_{i,j}\neq 0$, we have $v +
\varepsilon E_{v(j),i} \in X_{w}$ so $\sum a_{i,j} E_{v(j),v(i)}$ is
in the span of the independent set of $E_{v(j),v(i)}$ already known to
be in $T_{v}(X_{w})$.
\end{proof}

\begin{cor}
$X_{w}$ is smooth at $v\in S_n$ if and only if 
\begin{align*}
\mathrm{dim} \ T_{v}(X_{w}) &:= \# \{(i<j)\colon \  vt_{ij} \leq w \} =l(w)
\end{align*}
or equivalently if and only if 
\begin{align*}
\# \{(i<j)\colon  v<vt_{ij} \leq w \} =l(w)-l(v).
\end{align*}
\end{cor}

\begin{ex}
  Consider the case $n=4$.  The Schubert variety $X_{4231}$ is not
  smooth at the point $v=2143$.  For all 6 transpositions,
  $vt_{ij}\leq w$, but $\ell(w)=5$.  Also, $6=\#\{t_{ij}\leq 4231\} =
  \mathrm{dim}\ T_{id}(X_{4231})>\ell(4231)=5$. See Figure
  \ref{Interval [id,4231]} to verify these statements.
  Similarly, one can check $X_{3412}$ is not smooth at $v=1324$ and is
  smooth at all $v'\leq w$ such that $v' \not \leq v$.    It follows
  that
\begin{align*}
&\mathrm{Sing}(X_{4231}) = X_{2143}\\
&\mathrm{Sing}(X_{3412}) = X_{1324}.
\end{align*}
Note that $3412$ is the reverse of $2143$ and $4231$ is the reverse of $1324$.  
All other Schubert varieties $X_{w}$ for $w$ in $S_{4}$ are smooth.

\begin{figure}[h]
\centering
\includegraphics[height=9.0cm]{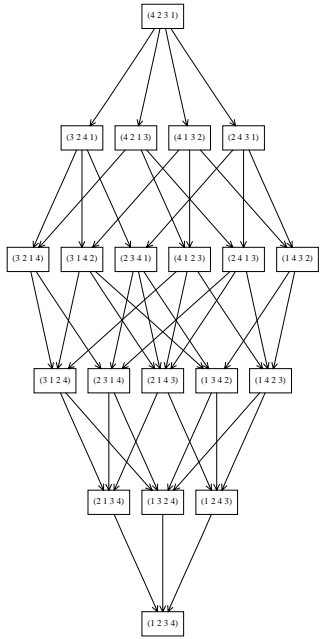}
\caption{The interval $[id,4231]$}\label{Interval [id,4231]}
\end{figure}
\end{ex}

\subsection{Bruhat graphs}
\begin{defn}
For a permutation $w$, the \textit{Bruhat graph} for $w$ is a graph
whose vertex set is $\{v\in S_n\colon v\leq w \}=[id,w]$ and there is an
edge between $v$ and $vt_{ij}$ if and only if both $v,vt_{ij} \leq w$.
\end{defn}

For example, the Bruhat graph of $w=4321$ is drawn in Figure
\ref{Bruhat graph w=432}.  Observe that the degree of $v$ (i.e. the
number of edges connected to $v$) in the Bruhat graph for $w$ is
$\mathrm{dim}\ T_{v}(X_{w})$.

The Bruhat graph for $w$ has a geometric interpretation: it is the
\textit{moment graph} of the Schubert variety $X_{w}$. Let $T\subset
GL_{n}$ be the set of invertible diagonal matrices, then the
permutation matrices in $GL_{n}/\glB$ are exactly the $T$-fixed points.
\begin{itemize}
\item[(i)] The permutations in $[id,w]$ are in bijection with the
  $T$-fixed points of $X_{w}$.
\item[(ii)] If $v,vt_{ij}\leq w$, then the edge between $v$ and
  $vt_{ij}$ in the Bruhat graph for $w$, is realized as the
  corresponding curve passing through the flags corresponding to $v$
  and $vt_{ij}$
\[
L_{v} = \{v+ z E_{v(j),i}\colon z \in \mathbb{C} \} \cup \{vt_{ij} \}   \approx \mathbb{P}^{1}.
\]
This curve is $T$-invariant, and pointwise fixed by a torus $T'\subset T$ of codimension $1$.
\end{itemize}
Schubert varieties are examples of \textit{GKM-spaces}
studied by Goresky-Kottwitz-MacPherson \cite{GKM} and
others.  
It turns out that much of the $T$-equivariant topology or geometry of
GKM spaces can be described in terms of their moment graph.
\begin{figure}[h]
\centering
\includegraphics[height=9cm]{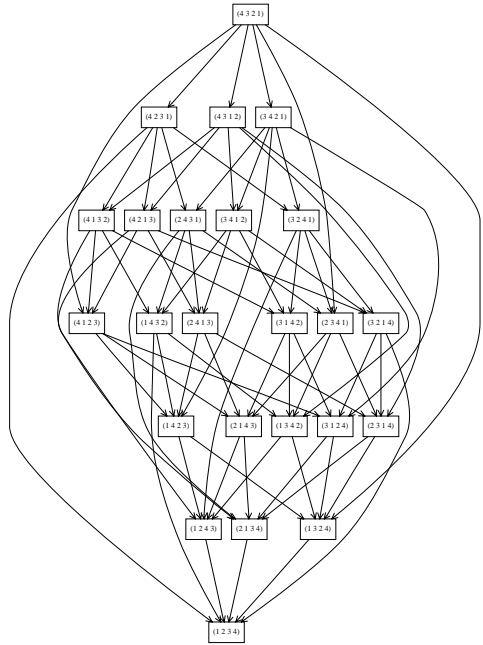}
\caption{The Bruhat graph of $w=4321$}\label{Bruhat graph w=432}
\end{figure}

\subsection{Lakshmibai-Sandhya Theorem}\label{Lakshmibai-Sandhya Theorem}
There exists a simple criterion for characterizing smooth Schubert
varieties using permutation pattern avoidance.  Pattern avoidance
first appeared in work by Knuth \cite{MR0378456}, Pratt
\cite{Pratt:1973} and Tarjan \cite{tarjan} related to computer sorting
algorithms in the 1960's and 1970's.  Today, many families of
permutations are characterized by pattern avoidance or variations on
that idea.  We discuss one of the key results that brought this
technique into the study of Schubert varieties.

Lakshmibai-Sandhya proved the following criterion for the singularity
of Schubert varieties in 1990. See also the mutually independent work  by Haiman
(unpublished), Ryan \cite{ryan}, and Wolper \cite{Wolper}.
\begin{thm}\label{Lak-San}
\emph{(}Lakshmibai-Sandhya \cite{Lak-San}\emph{)} 
$X_{w}$ is singular if and only if
  $w$ has a subsequence with the same relative order
  as 3412 or 4231.
\end{thm}
 
More generally, given any sequence of distinct real numbers $r_1\ldots
r_m$ define $fl(r_1\ldots r_m)$ to be the permutation $v \in S_m$ such
that $r_i < r_j$ if and only if $v_i<v_j$.  Recall that a permutation
is uniquely defined by its inversion set, so this condition uniquely
defines v.  The $fl$ operator \textit{flattens} the sequence.  Then, a
permutation $w=w_1w_2\ldots w_n \in S_n $ \textit{contains} a pattern
$v=v_1v_2\ldots v_m \in S_m$ for $m<n$ if there exists $i_1 < i_2 <
\ldots < i_m$ such that $fl(w_{i_1} w_{i_2} \ldots w_{i_m}) =v$.
Otherwise, $w$ \textit{avoids} $v$.

\begin{ex}
  The permutation $w=6 2 5 4 3 1$ contains the subsequence $6241$
  which flattens into a $4231$-pattern. Hence, $X_{625431}$ is
  singular. Also, $w=6 1 2 5 4 3$ avoids the patterns $4231$ and
  $3412$ which implies that $X_{612543}$ is non-singular.
\end{ex}

Let us sketch one approach to proving Theorem~\ref{Lak-San} by
applying Theorem~\ref{geom of Schubert 150}.  Say $w$ contains a
$3412$ or $4231$ pattern in positions $i_1<i_2<i_3<i_4$.  Let $v$ be
the permutation obtained from $w$ by rearranging the numbers $w_{i_1}
w_{i_2} w_{i_3} w_{i_4} $ according to the pattern for the
corresponding singular locus in $S_4$.  Specifically, if $w_{i_1}
w_{i_2} w_{i_3} w_{i_4} $ is a $4231$ then replace $w_{i_1} w_{i_2}
w_{i_3} w_{i_4} $ by the $2143$ pattern $w_{i_2} w_{i_4} w_{i_1}
w_{i_3} $ in the same positions.  If $w_{i_1} w_{i_2} w_{i_3} w_{i_4}
$ is a $3412$ then replace $w_{i_1} w_{i_2} w_{i_3} w_{i_4} $ by the
$1324$ pattern $w_{i_3} w_{i_1} w_{i_4} w_{i_2} $ in the same
positions.  For example, if $w=625431$ and we use the $6241$ instance
of the pattern $4231$, then $v=215634$ which contains a $2143$ pattern
among the values $1,2,4,6$.

We claim that $X_w$ is singular at the point $v$ by construction.  The
proof proceeds by comparing $\ell(w) - \ell(v)$ with the number of
$t_{ij}$ such that $v<vt_{ij}\leq w$.  For $i,j \in
\{i_1,i_2,i_3,i_4\}$, we know there will be strictly more such
transpositions than the length difference in these positions.  A key
lemma now states that if two permutations $v$ and $w$ agree in
position $i$, then $v \leq w$ if and only if $fl(v_1\ldots
\widehat{v_i} \ldots v_n) \leq fl(w_1\ldots \widehat{w_i} \ldots w_n)$
\cite[Lemma 2.1]{B3}.   This follows from looking at the rank tables of
two permutations differing by a transposition.  Next, note that
$vt_{ij}$ and $w$ differ in at most 6 positions.  Thus, by a computer
verification on permutations of length 6 one can show that
$$
\#\{t_{ij} \colon  v < vt_{ij} \leq w\} > \ell(w) -\ell(v).
$$

In the other direction, assume that $w$ avoids the patterns $4231$ and
$3412$.  Lakshmibai and Sandhya show that avoiding these patterns is
equivalent to an equidimensionality property of certain projections
which 
implies smoothness.

Haiman's proof also contained the following enumerative formula as a
corollary.  Since his paper was never published, it wasn't until 2007
that this result had a proof in the literature due to  Bousquet-M\'elou and Butler.

\begin{cor}\cite{BousquetMelou-Butler}\label{cor:smooth.perms}
  There is a closed form for the generating function for the sequence $v_n$ counting the number of smooth Schubert varieties for $w \in S_n$ :
\begin{align}\label{gf:haiman}
V(t) &= \ \frac{1-5t+3t^{2}+t^{2}\sqrt{1-4t}}{1-6t +8t^{2} -4 t^{3}}\\ \notag
 &= \ 1+ t+2t^{2}+6t^{3}+22t^{4}+88t^{5}+366t^{6}+1552t^{7}+6652t^{8}+
 O(t^{9}).
\end{align}
\end{cor}

Note that by the Lakshmibai-Sandhya theorem, testing for smoothness of
Schubert varieties can be done naively in polynomial time, $O(n^{4})$,
based on the characterization of avoiding 3412 and 4231.  As we
pointed out in the introduction, the Guillemot-Marx
\cite{Guillemot.Marx} construction leads to a linear time algorithm in
$n$ for testing if a permutation in $S_n$ contains either a 3412 or
4231 pattern.

Historically, there were some incremental results leading up to the
linear time algorithm to detect pattern avoidance by Guillemot and
Marx.  These other algorithms might still have useful applications, so
we mention a couple of them here.  In \cite{madras-liu}, Madras and
Liu study the 4231-avoiding permutations.  They point out that using
Knuth's original characterization of stack-sortable permutations in
linear time, one can find a $4231$ pattern in $O(n^2)$ time.  In fact,
Albert-Aldred-Atkinson-Holton show that every length 4 pattern can be
detected in $O(n \mathrm{log} n)$ time \cite{Albert01algorithmsfor}.

\sectionspace
\section{10 Pattern Avoidance Properties}\label{s:properties}
In this section, we exhibit the ubiquity of pattern avoidance as a
tool to characterize important properties in Schubert geometry and
related areas.  We give 10 distinct properties which are characterized
by pattern avoidance.  Each property will have a description in terms
of avoiding certain patterns.  Often these permutation families have
other distinguishing features as well.

The first family of permutations defined by pattern avoidance is the
3412- and 4231-avoiding permutations appearing in the
Lakshmibai-Sandhya Theorem.  It is a family rich in structure.  For
the record, we state all the properties equivalently characterized by
these two patterns.  The history, citations, and some definitions
follow the statement.

\begin{PAP}
\textit{\!\!
The following are equivalent for $w \in S_n$.
\begin{enumerate}
\item The one-line notation for $w$ avoids 3412 and 4231.
\item $X_{w}$ is smooth.
\item $\ell(w) = \#\{t_{ij}\leq w \}$.
\item \label{i:graph} The Bruhat graph for $w$ is regular and every vertex has degree $\ell(w)$.
\item \label{i:palindromic} The Poincar\'{e} polynomial for $w$, $P_{w}(t)= \sum_{v \leq w} t^{l(v)}$ is palindromic.
\item \label{i:factor} The Poincar\'{e} polynomial for $w$ factors as
\[
P_{w}(t)=\prod_{i=1}^{k}(1+t+t^{2}+\dots +t^{e_{i}})
\]
for some positive integers $\{e_1, e_2, ..., e_k\}$ such that $\ell(w)
= \sum e_i$.  
\item \label{i:hyper} 
  The Poincar\'{e} polynomial $P_w(t)$ is equal to the generating function
  $R_w(t)$ for the number of regions $r$ in the complement of the
  inversion hyperplane arrangement weighted by the distance of each
  region to the fundamental region.  In symbols,
\[
R_{w}(t)=\sum_r t^{d(r)} = \sum_{v \leq w} t^{l(v)} = P_{w}(t).
\]
Here, $d(r)$ is the number of hyperplanes crossed in a walk starting
at the fundamental region and going to the region $r$.
\item \label{i:hyper.free} 
The inversion arrangement for $w$ is free and the number of
  chambers of the arrangement is equal to the size of the Bruhat
  interval $[id,w]$.
\item \label{i:KL} 
   The Kazhdan-Lusztig polynomial $P_{x,w}(t)= 1$ for all $x\leq w$.
\item \label{i:KL.id} 
   The Kazhdan-Lusztig polynomial $P_{id,w}(t)= 1$.  
\end{enumerate}
}
\end{PAP}

We have already discussed the equivalence of the first three items.
Items~(\ref{i:graph}), (\ref{i:palindromic}), and~(\ref{i:KL.id}) are due to
Carrell and Peterson \cite{carrell94}.  Note, Carrell is the sole
author on the paper cited, but he always acknowledges Peterson as a
collaborator on this work so we give them both credit.  The term
\textit{palindromic} refers to the sequence of coefficients of the
polynomial, so the coefficient of $t^i$ equals the coefficient of
$t^{\ell(w)-i}$ in a palindromic Poincar\'{e} polynomial.

Item~(\ref{i:factor}) about factoring Poincar\'{e} polynomials is due
to Gasharov \cite{gasharov97}.  This factorization implies that the
geometry of smooth Schubert varieties has particularly nice structure
in terms of iterated fiber bundles over Grassmannians
\cite{GR2000,richmond.slofstra.2012,richmond.slofstra.2012,ryan,Wolper}.

\begin{ex} The permutation $w=4321$ avoids the patterns $3412$ and
  $4231$.   It has a palindromic Poincar\'e polynomial that also
  factors nicely, 
\begin{align*}
P_{4321}(t) 
&= (1+t)(1+t+t^{2})(1+t+t^{2}+t^{3})\\
&= 1+3t+5t^2+6t^3+5t^4+3t^5+t^6.
\end{align*}
\end{ex}

\begin{ex}  The permutation $3412$ is one of the two cases in $S_4$
  where the Poincar\'e polynomial does not have the nice
  factorization, nor the palindromic property.  Here 
$$P_{3412}(t) = 1+3t+5t^2+4t^3+t^4.$$
\end{ex}
\vspace{0.1in}

Item~(\ref{i:hyper}) about the inversion hyperplane arrangement is due
to Oh-Postnikov-Yoo \cite{OPY}.  This arrangement is given by the
collection of hyperplanes defined by $x_{i}-x_{j}=0$ for all $i<j$
such that $w(i)>w(j)$.  This generalizes the notion of the Coxeter
arrangement of type $A_{n-1}$ given by all the hyperplanes
$x_{i}-x_{j}=0$ for all $i<j$, so it is the inversion arrangement for
$w_0$.  The Coxeter arrangement has $n!$ regions corresponding to all
the permutations.  In this case, the statistic $d(w)$ equals
$\ell(w)$.  Note no explicit bijective proof of Item~(\ref{i:hyper}) is known.
The inversion arrangement comes up again in Property~5 below.

Item~(\ref{i:hyper.free}) is due to Slofstra \cite{slofstra.2013}.
Here a central hyperplane arrangement in a Euclidean space $V$ is said
to be \textit{free} if the module of derivations of the complexified 
arrangement is free as a module over the polynomial ring
$\mathbb{C}[V_{\mathbb{C}}]$.  We refer the reader to this paper for
more background.  Note it also gives an algebraic interpretation for
the generalized exponents $e_1, e_2, \ldots, e_k$ in terms of degrees
of a homogeneous basis for the module of derivations.


Items~(\ref{i:KL}) and~(\ref{i:KL.id}) concern the Kazhdan-Lusztig
polynomials \cite{k-l}.  These polynomials play an important role in
the study of the singularities of Schubert varieties and in
representation theory.  We recall the definitions here, highlight
some important developments,  and
refer the reader to the textbooks by Humphreys \cite{Hum} and
Bj\"orner-Brenti \cite{b-b} for more details.

The \textit{Hecke algebra} $\mathcal{H}$ associated with $S_n$ is an
algebra over $\Z[q^{\frac{1}{2}},q^{\frac{-1}{2}}]$ generated by
$\{T_i\colon 1\leq i\leq n-1\}$ with
the relations 
\begin{enumerate}
\item $(T_{i})^{2} = (q-1) T_{i} + q,$
\item $T_{i}T_{j} = T_{j} T_{i}$ \hspace{.1in} if $|i-j|>1,$ 
\item $T_{i}T_{i+1}T_{i} = T_{i+1} T_{i} T_{i+1}$   for all $1\leq i<n$.  
\end{enumerate}
This definition is patterned after the definition of the symmetric
group $S_n$ written in terms of its generating set of adjacent
transpositions and their relations. In fact, if we take the
specialization $q=1$, then the resulting algebra is the group algebra
of $S_n$.  The relations (2) and (3) are called the \textit{braid
  relations}.  The braid relations imply that $T_{w} = T_{i_{1}}
T_{i_{2}}\cdots T_{i_{p}}$ is well defined for any reduced expression
$w=s_{i_{1}} s_{i_{2}} \dots s_{i_{p}}$.  We will use the notation
$T_{id} =1 \in \mathcal{H}$ for the empty product of generators.  

An easy observation is that $\{T_{w} \colon w \in S_n \}$ is a linear
basis for $\mathcal{H}$ over $\Z[q^{\frac{1}{2}},q^{\frac{-1}{2}}]$ .
One can also observe that the $T_{w}$'s are invertible over $\Z[q,
q^{-1}]$ which can be see as follows.  First check that $\left(T_{i}
\right)^{-1} = q^{-1} T_{i} - \left(1- q^{-1} \right)$ by multiplying
by $T_{i}$ and using the stated relations.  Then, we have
$(T_{w})^{-1} = (T_{i_{p}})^{-1} \cdots (T_{i_{1}})^{-1}$ for a
reduced expression $w=s_{i_{1}} s_{i_{2}} \dots s_{i_{p}}$.

Next, let us review the \textit{Kazhdan-Lusztig involution}. Consider
the $\Z$-linear transformation $i\colon\mathcal{H}\rightarrow \mathcal{H}$
sending $T_{w}$ to $(T_{w^{-1}})^{-1}$ and $q$ to $q^{-1}$,
respectively.

\begin{thm}\label{Kazhdan-Lusztig basis}
\emph{(}Kazhdan-Lusztig \cite{k-l}\emph{)}
There exists a unique basis $\{C'_{w} \colon w \in S_n\}$ for the Hecke algebra $\mathcal{H}$ over $\mathbb{Z}[q^{\frac{1}{2}},q^{\frac{-1}{2}}]$  such that 
\begin{enumerate}
\item[(i)] $i( C'_{w} ) = C'_{w}$.
\item[(ii)] The change of basis matrix from $\{C'_w\}$ to $\{T_w\}$
is upper triangular when the elements of $S_n$ are listed in a total
order respecting Bruhat order, and the expansion coefficients $P_{x,w}(q)$ in
\[
C_w' = q^{-\frac{1}{2} \ell(w)} \sum_{x\leq w} P_{x,w}(q) \ T_x 
\]
have the properties $P_{w,w}=1$ and for all $x<w$, \ $P_{x,w}(q) \in \mathbb{Z}[q]$
with degree at most
\[
\frac{\ell(w)-\ell(x)-1}{2}.
\]
\end{enumerate}
\end{thm}

The basis $\{C'_{w} \colon w \in S_n\}$ is called  the \textit{Kazhdan-Lusztig
  basis} for $\mathcal{H}$, and $P_{x,w}(q)$ is the
\textit{Kazhdan-Lusztig polynomial} for $x,w\in S_n$.  This theorem
easily generalizes to all Coxeter groups for the reader familiar with 
that topic.  

\begin{ex}
  We exhibit some computations with the Kazhdan-Lusztig basis indexed by
  permutations with the aid of Theorem \ref{Kazhdan-Lusztig basis}.
  First, it is easy to see
\begin{align*}
C'_{s_{i}} &= q^{-\frac{1}{2}} (1+ T_{i}) = q^{\frac{1}{2}} (1+ T_{i}^{-1}). 
\end{align*}
Then, for $i \neq j$, the computation 
\begin{align*}
C'_{s_{i}}C'_{s_{j}} &= q^{-1} (1+ T_{i})(1+ T_{j})= q^{-1} (1 + T_{i} + T_{j} + T_{i}T_{j})
\end{align*}
shows that $C'_{s_{i} s_{j}}=C'_{s_{i}}C'_{s_{j}}$ for $i \neq j$.
Also, in the computation
\begin{align*}
C'_{s_{1}}C'_{s_{2}}C'_{s_{1}} &= q^{-\frac{3}{2}} (1+ T_{1})(1+ T_{2}) (1+ T_{1}) \\
                &= q^{-\frac{3}{2}} (1 + 2T_{1} + T_{2} + T_{1}T_{2} + T_{2}T_{1} + 
T_{1}^{2} + T_{1}T_{2}T_{1})\\
	        &= q^{-\frac{3}{2}} (1 + 2T_{1} + T_{2} + T_{1}T_{2} + T_{2}T_{1} 
                    +((q-1) T_{1} + q ) + T_{1}T_{2}T_{1}),
\end{align*}
one notices that $qT_{1} + q$ which comes from $T_1^2$ should not
appear for $C'_{s_{1}s_{2}s_{1}}$ because the degree of the polynomial
coefficient of $T_1$ and $T_{id}$ are too large.  We need a correction
term.  Since $C'_{s_{i}} = q^{-\frac{1}{2}} (1+ T_{i})$ one can check
that $C'_{s_{1}s_{2}s_{1}} = C'_{s_{1}}C'_{s_{2}}C'_{s_{1}} -
C'_{s_{1}}$ by Theorem~\ref{Kazhdan-Lusztig basis}. 
\end{ex}

\begin{ex}
 If $i_1,\cdots,i_k\in [n-1]$ are distinct, then one can
check that
\begin{align*}
C'_{s_{i_1}\cdots s_{i_k}} = C'_{s_{i_1}} \cdots C'_{s_{i_k}}.
\end{align*}
More generally, a permutation $w \in S_{n}$ is called \textit{Deodhar}
if $C'_{w}=C'_{s_{i_{1}}} C'_{s_{i_{2}}}\cdots C'_{s_{i_{p}}} $ for
some reduced expression $w=s_{i_{1}}s_{i_{2}}\cdots s_{i_{p}}$.  We
will return to the Deodhar permutations in Property 6.
\end{ex}

\begin{ex}
  The Kazhdan-Lusztig polynomials $P_{\mathrm{id}, w}$ for $w\in S_5$
  are completely determined from the following table and the fact that
  $P_{\mathrm{id}, w}=1$ if and only if $w$ is 3412 and 4231
  avoiding.  
\begin{center}
\begin{tabular}{|c|c|}\hline
$w$ &  $P_{\mathrm{id},w}$\\
\hline
$\displaystyle
\begin{array}{lll}
(1 4 5 2 3)&
(1 5 3 4 2)&
(2 4 5 1 3)\\
(2 5 3 4 1)&
(3 4 1 2 5)&
(3 4 1 5 2)\\
(3 5 1 2 4)&
(3 5 1 4 2)&
(3 5 2 4 1)\\
(3 5 4 1 2)&
(4 1 5 2 3)&
(4 2 3 1 5)\\
(4 2 3 5 1)&
(4 2 5 1 3)&
(4 2 5 3 1)\\
(4 3 5 1 2)&
(4 5 1 3 2)&
(4 5 2 1 3)\\
(5 1 3 4 2)&
(5 2 3 1 4)&
(5 2 4 1 3)\\
(5 2 4 3 1)&
(5 3 1 4 2)&
(5 3 2 4 1)\\
(5 3 4 2 1)&
(5 4 2 3 1)&
\end{array}$ &  $q+1 $\\
\hline
$ \displaystyle \begin{array}{ll}
(3 4 5 1 2)&
(4 5 1 2 3) \\
(4 5 2 3 1)&
(5 3 4 1 2)
\end{array} $&  $2q + 1$\\
\hline
$ \displaystyle  \begin{array}{l}
(5 2 3 4 1)
\end{array}$ &  $ q^2 + 2q + 1$\\
\hline
$\begin{array}{l}
(4 5 3 1 2)
\end{array}$ &  $q^2 + 1$\\
\hline
\end{tabular}
\end{center}
\end{ex}

\bigskip

The reader might notice that all coefficients of Kazhdan-Lusztig
polynomials shown so far are non-negative integers.  In their 1979
paper, this property was stated as a conjecture for all
Kazhdan-Lusztig polynomials. In 1980, Kazhdan and Lusztig proved their
own conjecture using \textit{intersection homology} as introduced by
Goresky and MacPherson in 1974, see \cite{g-m} as a good starting
point for that theory.

\begin{thm}
\emph{(}Kazhdan-Lusztig \cite{K-L2}\emph{)}
If $W$ is a Weyl group or affine Weyl group then 
\[
P_{x,w}(q) = \sum \mathrm{dim} \mathcal{IH}_{x}^{i}(X_{w}) \ q^{i}.
\]
\end{thm}

\begin{cor}
The coefficients of $P_{x,w}(q)$ are non-negative integers with constant term 1.
\end{cor}

The big news in Kazhdan-Lusztig theory is the recent proof that all
Kazhdan-Lusztig polynomials for all Coxeter groups have non-negative
integer coefficients.  This proof is due to Elias and
Williamson~\cite{Elias-Williamson}.  They give an algebraic structure (Soergel bimodules)
which plays the same role as intersection homology of Schubert
varieties in the original proof.

As stated in Property 1, Kazhdan-Lusztig polynomials can be used to
determine smoothness of Schubert varieties (in type A).  There are
several other interesting properties of Kazhdan-Lusztig polynomials
that have emerged since they were defined in 1979.  We cover some of
them here and recommend the Wikipedia page \cite{wiki:kl} 
for a very nice survey.  
\begin{enumerate}
\item 
  In 1981, Beilinson--Bernstein, and independently
  Brylinski--Kashiwara, proved another important conjecture due to
  Kazhdan and Lusztig.  They showed that the multiplicities which
  appear when expressing the formal character of a Verma module in
  terms of the formal character for the corresponding simple highest
  weight module are determined by evaluating Kazhdan-Lusztig
  polynomials at $q=1$ (\cite{BeilBern, BryKash}).


\item 

  The coefficients of Kazhdan-Lusztig polynomials are increasing as
  one goes down in Bruhat order, while keeping the second index fixed.
  Specifically, if $x\leq y \leq w$, then $\displaystyle
  \mathrm{coef}_{q^{k}} P_{x,w}(q) \geq \mathrm{coef}_{q^{k}}
  P_{y,w}(q)$.  This monotonicity property was first published in 1988
  by Ron Irving \cite{Irving}.  Irving's proof is based on the socle
  filtration of a Verma module.  In 2001, Braden and MacPherson gave a
  different proof using intersection homology \cite[Cor. 3.7]{B-M}.

\item 
  Every polynomial with constant term 1 and nonnegative integer
  coefficients is the Kazhdan-Lusztig polynomial of some pair of
  permutations.  This is due to Patrick Polo, published in 1999
  \cite{polo}.  He gives an explicit construction of the pair of
  permutations for a given polynomial.  This was a surprising result
  because from the small data that we can compute, say for $n\leq 8$,
  the polynomials seem quite special.  They must get increasingly
  complex as $n$ grows.

\item Let $\mu (x,w) $ be the coefficient of
  $q^{\frac{\ell(w)-\ell(x)-1}{2}}$ in $P_{x,w}$.  Note, $\mu(x,w)$
  can be 0.  For $x,w \in S_{9}$,\ $\mu (x,w) \in \{0,1 \}$.
  MacLarnen and Warrington found an example in $S_{10}$ where $\mu
  (x,w) = 5$ \cite{MW02}.  Prior to their publication in 2003, this
  was referred to as the ``0-1 Conjecture for Kazhdan-Lusztig
  polynomials.''  This again demonstrates the increasing complexity
  as $n$ grows.  The reader might be wondering how anyone could have
  believed the 0-1 Conjecture after seeing Polo's theorem in (3).
  However, Polo's theorem does not contradict the 0-1 Conjecture because
  in his construction the length difference between $w$ and $x$ is
  large enough that the leading term in $P_{x,w}(q)$ is typically not
  the $\mu$-coefficient.

\item 
  There exists a formula for $P_{x,w}(q)$ which only depends on the
  abstract interval $[id,w]$ in Bruhat order.  See the work of du
  Cloux (2003) \cite{ducloux.2004}, Brenti (2004) \cite{brenti04} and
  Brenti-Caselli-Marietti (2006) \cite{BCM.2006}.  
\end{enumerate}

There are two interesting but difficult open problems in
Kazhdan-Lusztig theory.  There are many partial answers to these
questions in the literature, but we don't know of a complete source at
this time.   Perhaps there is a need for someone to start a wiki page.

\begin{ques}\label{q:lusztig.interval}(Lusztig)
  Can one compute $P_{x,w}(q)$ using only the abstract poset given by
  the interval $[x,w]$ in Bruhat order?  In other words, $P_{u,v}(q) =
  P_{x,y}$ whenever $[u,v]$ and $[x,y]$ are isomorphic as posets.
\end{ques}

\begin{ques}\label{q:kl.poly}
  Can one compute the coefficients of $P_{x,w}(q)$ by counting
  combinatorially defined objects?  
\end{ques}

\bigskip

The next pattern property connects the 3412 and 4231 patterns to the
determination of the singular locus of a Schubert variety.  Recall
from Section~\ref{geom of Schubert}, the singular locus of a Schubert
variety $X_w$ is a union of Schubert varieties $X_v$ with $v<w$.  Thus
to determine the irreducible components of the singular locus, we just
need to give the maximal permutations $v <w$ such that $v$ determines a
singular point in $X_w$.

\begin{PAP}\textit{\!\!
\emph{(}Billey-Warrington, Manivel, Kassel-Lascoux-Reutenauer, and
Cortez \cite{BW-sing, manivel, klr, cortez}\emph{)} 
$X_{v}$ is an irreducible component of the singular locus of $X_{w}$ if and only if
\[
v=w\cdot \text{(1-cycle permutation)}
\]
corresponding to a 4231 or 3412 or 45312 pattern from
Figure~\ref{f:sing.locus} such that the shaded region contains no
additional 1's except in the 45312 case where they must appear in the
central region in decreasing order.
Here $\circ$'s denote
  1's in $w$, $\bullet$'s denote 1's in $v$.
}
\end{PAP}

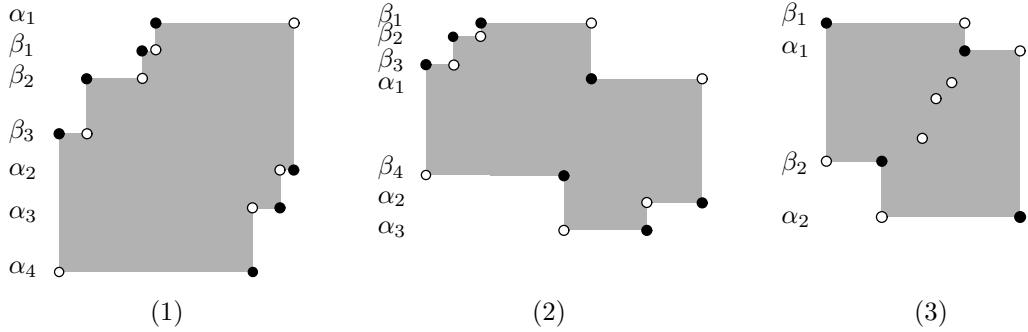
\begin{figure}[h]
\centering
\unitlength 0.1in
\begin{picture}( 53.2100, 15.8300)(  4.0000,-21.9000)
\put(11.4000,-23.6000){\makebox(0,0)[lb]{(1)}}%
\put(31.4000,-23.6000){\makebox(0,0)[lb]{(2)}}%
\put(51.4000,-23.6000){\makebox(0,0)[lb]{(3)}}%
%
\special{pn 8}%
\special{sh 0.300}%
\special{pa 1170 770}%
\special{pa 1894 770}%
\special{pa 1894 1540}%
\special{pa 1170 1540}%
\special{pa 1170 770}%
\special{ip}%
%
\special{pn 8}%
\special{sh 0.300}%
\special{pa 1098 916}%
\special{pa 1822 916}%
\special{pa 1822 1640}%
\special{pa 1098 1640}%
\special{pa 1098 916}%
\special{ip}%
%
\special{pn 8}%
\special{sh 0.300}%
\special{pa 808 1062}%
\special{pa 1532 1062}%
\special{pa 1532 1786}%
\special{pa 808 1786}%
\special{pa 808 1062}%
\special{ip}%
%
\special{pn 8}%
\special{sh 0.300}%
\special{pa 664 1350}%
\special{pa 1388 1350}%
\special{pa 1388 2074}%
\special{pa 664 2074}%
\special{pa 664 1350}%
\special{ip}%
%
\special{pn 8}%
\special{sh 0.300}%
\special{pa 954 1350}%
\special{pa 1678 1350}%
\special{pa 1678 2074}%
\special{pa 954 2074}%
\special{pa 954 1350}%
\special{ip}%
%
\special{pn 8}%
\special{sh 0.300}%
\special{pa 1098 1062}%
\special{pa 1822 1062}%
\special{pa 1822 1740}%
\special{pa 1098 1740}%
\special{pa 1098 1062}%
\special{ip}%
%
\special{pn 8}%
\special{sh 1.000}%
\special{ar 1170 770 26 26  0.0000000 6.2831853}%
%
\special{pn 8}%
\special{sh 1.000}%
\special{ar 1098 916 26 26  0.0000000 6.2831853}%
%
\special{pn 8}%
\special{sh 1.000}%
\special{ar 808 1062 26 26  0.0000000 6.2831853}%
%
\special{pn 8}%
\special{sh 1.000}%
\special{ar 664 1350 26 26  0.0000000 6.2831853}%
%
\special{pn 8}%
\special{sh 0}%
\special{ar 664 2074 24 24  0.0000000 6.2831853}%
%
\special{pn 8}%
\special{sh 1.000}%
\special{ar 1678 2074 24 24  0.0000000 6.2831853}%
%
\special{pn 8}%
\special{sh 0}%
\special{ar 1676 1738 26 26  0.0000000 6.2831853}%
%
\special{pn 8}%
\special{sh 1.000}%
\special{ar 1820 1738 26 26  0.0000000 6.2831853}%
%
\special{pn 8}%
\special{sh 0}%
\special{ar 1820 1540 26 26  0.0000000 6.2831853}%
%
\special{pn 8}%
\special{sh 1.000}%
\special{ar 1892 1540 26 26  0.0000000 6.2831853}%
%
\special{pn 8}%
\special{sh 0}%
\special{ar 1894 770 26 26  0.0000000 6.2831853}%
%
\special{pn 8}%
\special{sh 0.300}%
\special{pa 3306 1422}%
\special{pa 3740 1422}%
\special{pa 3740 1856}%
\special{pa 3306 1856}%
\special{pa 3306 1422}%
\special{ip}%
%
\special{pn 8}%
\special{sh 0.300}%
\special{pa 3596 1278}%
\special{pa 4030 1278}%
\special{pa 4030 1712}%
\special{pa 3596 1712}%
\special{pa 3596 1278}%
\special{ip}%
%
\special{pn 8}%
\special{sh 0.300}%
\special{pa 2872 770}%
\special{pa 3306 770}%
\special{pa 3306 1206}%
\special{pa 2872 1206}%
\special{pa 2872 770}%
\special{ip}%
%
\special{pn 8}%
\special{sh 0.300}%
\special{pa 2726 842}%
\special{pa 3162 842}%
\special{pa 3162 1278}%
\special{pa 2726 1278}%
\special{pa 2726 842}%
\special{ip}%
%
\special{pn 8}%
\special{sh 0.300}%
\special{pa 3596 1062}%
\special{pa 4030 1062}%
\special{pa 4030 1494}%
\special{pa 3596 1494}%
\special{pa 3596 1062}%
\special{ip}%
%
\special{pn 8}%
\special{sh 0.300}%
\special{pa 3306 1062}%
\special{pa 3740 1062}%
\special{pa 3740 1494}%
\special{pa 3306 1494}%
\special{pa 3306 1062}%
\special{ip}%
%
\special{pn 8}%
\special{sh 0.300}%
\special{pa 2584 988}%
\special{pa 3016 988}%
\special{pa 3016 1566}%
\special{pa 2584 1566}%
\special{pa 2584 988}%
\special{ip}%
%
\special{pn 8}%
\special{sh 0.300}%
\special{pa 2872 988}%
\special{pa 3306 988}%
\special{pa 3306 1566}%
\special{pa 2872 1566}%
\special{pa 2872 988}%
\special{ip}%
%
\special{pn 8}%
\special{sh 0.300}%
\special{pa 3016 770}%
\special{pa 3450 770}%
\special{pa 3450 1206}%
\special{pa 3016 1206}%
\special{pa 3016 770}%
\special{ip}%
%
\special{pn 8}%
\special{sh 0.300}%
\special{pa 4680 770}%
\special{pa 5404 770}%
\special{pa 5404 1494}%
\special{pa 4680 1494}%
\special{pa 4680 770}%
\special{ip}%
%
\special{pn 8}%
\special{sh 0.300}%
\special{pa 4970 916}%
\special{pa 5694 916}%
\special{pa 5694 1786}%
\special{pa 4970 1786}%
\special{pa 4970 916}%
\special{ip}%
%
\special{pn 8}%
\special{sh 1.000}%
\special{ar 2872 770 26 26  0.0000000 6.2831853}%
%
\special{pn 8}%
\special{sh 1.000}%
\special{ar 2726 842 24 24  0.0000000 6.2831853}%
%
\special{pn 8}%
\special{sh 1.000}%
\special{ar 2584 988 26 26  0.0000000 6.2831853}%
%
\special{pn 8}%
\special{sh 1.000}%
\special{ar 3450 1062 26 26  0.0000000 6.2831853}%
%
\special{pn 8}%
\special{sh 1.000}%
\special{ar 3740 1856 26 26  0.0000000 6.2831853}%
%
\special{pn 8}%
\special{sh 1.000}%
\special{ar 4030 1712 26 26  0.0000000 6.2831853}%
%
\special{pn 8}%
\special{sh 0}%
\special{ar 3450 770 26 26  0.0000000 6.2831853}%
%
\special{pn 8}%
\special{sh 0}%
\special{ar 2584 1566 24 24  0.0000000 6.2831853}%
%
\special{pn 8}%
\special{sh 0}%
\special{ar 3306 1856 26 26  0.0000000 6.2831853}%
%
\special{pn 8}%
\special{sh 0}%
\special{ar 4030 1062 26 26  0.0000000 6.2831853}%
%
\special{pn 8}%
\special{sh 0}%
\special{ar 5694 916 26 26  0.0000000 6.2831853}%
%
\special{pn 8}%
\special{sh 0}%
\special{ar 5404 770 26 26  0.0000000 6.2831853}%
%
\special{pn 8}%
\special{sh 0}%
\special{ar 4680 1494 26 26  0.0000000 6.2831853}%
%
\special{pn 8}%
\special{sh 0}%
\special{ar 4970 1786 28 28  0.0000000 6.2831853}%
%
\special{pn 8}%
\special{sh 1.000}%
\special{ar 4680 770 26 26  0.0000000 6.2831853}%
%
\special{pn 8}%
\special{sh 1.000}%
\special{ar 5694 1786 28 28  0.0000000 6.2831853}%
%
\special{pn 8}%
\special{sh 1.000}%
\special{ar 4970 1494 26 26  0.0000000 6.2831853}%
%
\special{pn 8}%
\special{sh 1.000}%
\special{ar 5404 916 26 26  0.0000000 6.2831853}%
%
\special{pn 8}%
\special{sh 0.300}%
\special{pa 2922 990}%
\special{pa 3354 990}%
\special{pa 3354 1568}%
\special{pa 2922 1568}%
\special{pa 2922 990}%
\special{ip}%
%
\special{pn 8}%
\special{sh 1.000}%
\special{ar 3306 1570 26 26  0.0000000 6.2831853}%
%
\special{pn 8}%
\special{sh 0}%
\special{ar 5182 1374 26 26  0.0000000 6.2831853}%
%
\special{pn 8}%
\special{sh 0}%
\special{ar 5254 1166 26 26  0.0000000 6.2831853}%
%
\special{pn 8}%
\special{sh 0}%
\special{ar 5336 1082 26 26  0.0000000 6.2831853}%
\put(4.0000,-7.7700){\makebox(0,0)[lb]{$\alpha_1$}}%
\put(4.0000,-9.4800){\makebox(0,0)[lb]{$\beta_1$}}%
\put(4.0000,-11.0900){\makebox(0,0)[lb]{$\beta_2$}}%
\put(4.0000,-13.9700){\makebox(0,0)[lb]{$\beta_3$}}%
\put(4.0000,-15.8600){\makebox(0,0)[lb]{$\alpha_2$}}%
\put(4.0000,-18.1000){\makebox(0,0)[lb]{$\alpha_3$}}%
\put(4.0000,-21.0000){\makebox(0,0)[lb]{$\alpha_4$}}%
\put(23.3000,-7.9000){\makebox(0,0)[lb]{$\beta_1$}}%
\put(23.3000,-8.8900){\makebox(0,0)[lb]{$\beta_2$}}%
\put(23.3300,-10.2800){\makebox(0,0)[lb]{$\beta_3$}}%
\put(23.3300,-11.3600){\makebox(0,0)[lb]{$\alpha_1$}}%
\put(23.3300,-15.8600){\makebox(0,0)[lb]{$\beta_4$}}%
\put(23.3300,-17.3000){\makebox(0,0)[lb]{$\alpha_2$}}%
\put(23.3300,-18.9100){\makebox(0,0)[lb]{$\alpha_3$}}%
\put(44.4500,-7.8600){\makebox(0,0)[lb]{$\beta_1$}}%
\put(44.4500,-9.4800){\makebox(0,0)[lb]{$\alpha_1$}}%
\put(44.4500,-15.6800){\makebox(0,0)[lb]{$\beta_2$}}%
\put(44.4500,-18.3800){\makebox(0,0)[lb]{$\alpha_2$}}%
%
\special{pn 8}%
\special{sh 0}%
\special{ar 1170 910 26 26  0.0000000 6.2831853}%
%
\special{pn 8}%
\special{sh 0}%
\special{ar 1100 1060 26 26  0.0000000 6.2831853}%
%
\special{pn 8}%
\special{sh 0}%
\special{ar 810 1350 26 26  0.0000000 6.2831853}%
%
\special{pn 8}%
\special{sh 0}%
\special{ar 2730 990 26 26  0.0000000 6.2831853}%
%
\special{pn 8}%
\special{sh 0}%
\special{ar 2870 840 26 26  0.0000000 6.2831853}%
%
\special{pn 8}%
\special{sh 0}%
\special{ar 3740 1710 26 26  0.0000000 6.2831853}%
\end{picture}%
\vspace{0.1in}
\caption{Patterns for the singular locus of $X_w$ in the 4231, 3412,
  and 45312 cases respectively.}\label{f:sing.locus}
\end{figure}

This result was found around 2000 by 7 authors in 4 papers, plus
Gasharov proved on direction of the conjecture \cite{gasharov.2001}
around the same time. It must have been ripe for discovery.  It
refined and proved a conjecture due to Lakshmibai and Sandhya
\cite{Lak-San}.  For the sake of history, we note that the authors of
\cite{BW-sing} were the first to report this result to Lakshmibai.

\begin{cor}
  The codimension of the singular locus of a Schubert variety $X_w$ is
  at least 3 for any $w \in S_n$.   
\end{cor}

The corollary is in fact true for all simply laced types. However, it
is not true in type $B_n$.  The codimension of the singular locus of a
Schubert variety in that case can be 2.

Inspired by the Lakshmibai-Sandhya Theorem and the construction of the
singular locus of a Schubert variety in Property 2, Woo and Yong
\cite{WooYong} defined the notion of interval pattern avoidance.
Given permutations $u <v \in S_m$ and $x<y \in S_n$ for $m<m$, say
$[u,v]$ \textit{interval pattern embeds} into $[x,y]$ provided
\begin{enumerate}
\item There exist indices $1\leq i_1< i_2 < \ldots < i_m \leq n$ such that
$\fl(x_{i_1}, \ldots, x_{i_m}) = u$  and $\fl(y_{i_1}, \ldots,
y_{i_m}) = v$ respectively.  
\item The permutations $x,y$ agree in all positions other than $1\leq
  i_1< i_2 < \ldots < i_m \leq n$.   
\item The Bruhat intervals $[u,v]$ and $[x,y]$ are isomorphic as
  posets.   
\end{enumerate}
In fact, if $x,y$ agree everywhere outside of the indices $1\leq i_1<
i_2 < \ldots < i_m \leq n$ and $u = \fl(x_{i_1}, \ldots, x_{i_m}),\ v=
fl(y_{i_1}, \ldots, y_{i_m})$ then $[u,v]$ interval embeds in $[x,y]$
if and only if $\ell(v)-\ell(u) = \ell(y)-\ell(x)$ \cite[Lemma
2.1]{WooYong}.  Furthermore, for all $w \in S_n$ such that $x<w<y$,
then $w$ agrees everywhere with $y$ outside of the sequence and
$[fl(w_{i_1}, \ldots, w_{i_m}), v]$ also interval embeds in $[x,y]$ \cite[Lemma
2.4]{WooYong}.   

Observe that the condition from Figure~\ref{f:sing.locus} that the
shaded region have no additional 1's in the permutation matrices
implies that the length $l(w)-l(v)$ is equal to the corresponding
length drop in each of the 4231, 3412 or 45312 cases.  Thus, the
maximal singular locus of a Schubert variety is determined by interval
pattern conditions.   

Another example of the power of interval pattern embeddings is the
following result supporting Question~\ref{q:lusztig.interval}.  More
examples will follow, but the reader is encouraged to see
\cite{WooYong} for more details.

\begin{thm}\cite[Cor. 6.3]{WooYong}\label{thm:woo.yong}  Suppose $[u,v] \subset S_m$
  interval pattern embeds into $[x,y] \in S_n$, then the
  Kazhdan-Lusztig polynomials $P_{u,v}(q) $ and $P_{x,y}(q) $ are
  equal.  
\end{thm}

\bigskip

Next, recall by a theorem due to Zariski that a variety $X$ is smooth if and only if the local
ring at every point is regular. 
A variety $X$ is \textit{factorial} at a point if the local ring at
that point is a unique factorization domain. Note that a smooth
variety is factorial at every point since any regular local ring is a
unique factorization domain.  The following property was conjectured by
Woo-Yong and proved by Bousquet-M\'elou and Butler in 2007.
\begin{PAP}
\textit{\!\!
\emph{(}Bousquet-M\'{e}lou-Butler \cite{BousquetMelou-Butler}\emph{)}
Let $w \in S_n$, then the following are equivalent.  
\begin{enumerate}
\item  The Schubert variety $X_{w}$ is factorial at every point.
\item The permutation  $w$ avoids $4231$ and $3\underline{41}2$ where $3\underline{41}2$
  means that the $4$ and $1$ must be adjacent in the one-line notation
  for $w$.
\item The permutation $w$ avoids $4231$, and for every $v<w$ differing
  in exactly 4 positions, the interval $[v,w]$ is not isomorphic to
  $[3142,3412]$.  Thus, one says $w$ interval avoids $[3142,3412]$.  
\end{enumerate}
 }
\end{PAP}

Compare the generating function below with
Corollary~\ref{cor:smooth.perms} which is the generating function for
the number of smooth Schubert varieties in $\flags$.

\begin{thm}\cite{BousquetMelou-Butler}  
There is a closed form for the generating function for the sequence $f_{n}$ counting the 
factorial Schubert varieties for $w \in S_n$:
\begin{align}\label{gf:factorial}
F(t) &= 
\frac{(1-t)(1-4t-2t^{2})-(1-5t)\sqrt{(1-4t)}}{2(1-5t+2t^2-t^3)} \\ \notag
 &= t + 2t^2 + 6t^3 + 22t^4 + 89t^5 + 379t^6 + 1661t^7 + 7405t^8 + \ldots .
\end{align}
\end{thm}

Note, the term $\sqrt{1-4t}$ appears in both \eqref{gf:haiman} and
\eqref{gf:factorial}.  This term is familiar in combinatorics because
it also appears in the generating function for the Catalan numbers,
$c_n=\frac{1}{n+1}\abechs{2n}{n}$.  In particular, as a power series
$\sqrt{1-4t} = 1+\sum_{n\geq 1} \frac{-2}{n} \abechs{2n-2}{n-1}t^{n} $
by Newton's generalized binomial theorem.  Thus, the generating
function for the Catalan numbers is
$$
\sum_{n\geq1} c_n t^n  = \frac{1-\sqrt{1-4t}}{2t}.  
$$

\bigskip
There exists a simple criterion for characterizing
Gorenstein Schubert varieties using modified pattern avoidance.
Recall that a variety $X$ is \textit{Gorenstein} if it is Cohen-Macaulay and its canonical sheaf is a line bundle. 
Woo and Yong characterized the Gorenstein condition by using  pattern avoidance.
\begin{PAP}
\textit{\!\!
\emph{(}Woo-Yong \cite{w-y}\emph{)}
A Schubert variety $X_{w}$ is Gorenstein if and only if the following two conditions are satisfied :
\begin{itemize}
\item[(i)] $w$ avoids $35142$ and $42513$ \textit{with Bruhat restrictions}  
$\{t_{15},t_{23} \}$ and $\{t_{15}, t_{34} \}$, and 
\item[(ii)] for each descent $d$ in $w$, the associated partition
$\lambda_{d}(w)$ has all of its inner corners on the same
antidiagonal.
\end{itemize}
 Later, Woo-Yong \cite[Thm. 6.6]{WooYong} also gave a
characterization of Gorenstein Schubert varieties in terms of an
interval pattern avoidance using an infinite number of intervals.   }
\end{PAP}

We note that in the paper \cite{w-y}, the theorem states that $w$
should avoid $31542$ and $24153$ which is twisted by $w_{0}$ from the
permutations written above.  The difference is that they are labeling
Schubert varieties in such a way that the codimension of $X_{w}$ is
$\ell(w)$ which works better for computing products of Schubert
classes.

 The proof of this result due to Woo and
Yong relates the Gorenstein property to Schubert classes for the flag
manifold and Monk's formula.  Since the topic of the conference in
Osaka is ``Schubert Calculus'', we outline this proof to show the
logical relationship.  The steps are due to Woo and Yong unless
otherwise mentioned.

\bigskip

\noindent
\textit{Sketch of proof.}
\begin{itemize}
\item[Step 1:]   Schubert varieties are all Cohen-Macaulay.  (Ramanathan, 1985)  

\item[Step 2:] Testing if $X_{w}$ is Gorenstein reduces to a
comparison using the Weil divisor class group and the Cartier class
group for $X_w$. (Brion, Knutson, Kumar) 

\item[Step 3:] The Weil divisor class group is generated by the set of
  all $[X_{v}] \in H^{*}(G/B)$ such that $v<w$ and
  $\ell(v)=\ell(w)-1$.  In this case we say $w$ covers $v$ in Bruhat
  order, denoted $v \lhd w$.  If $v \lhd w$, then $w=vt_{ij}$ but
  $t_{i,j}$ does not need to be an adjacent transposition.   

\item[Step 4:] The Cartier class group is generated by $[X_{w_{0}s_{i}}][X_{w}] $ and 
\[
[X_{w_{0}s_{i}}][X_{w}] = \sum [X_{v}]
\]
summed over all $v=wt_{ab}\colon a\leq i <b, \ \ell(v) = \ell(w)-1$ by
Monk's formula.  

\item[Step 5:]  The Schubert variety $X_{w}$ is Gorenstein if and only if there exists an integral
solution $(\alpha_{1}, \dots , \alpha_{n-1})$ to 
\[
\sum_{i=1}^{n-1} \alpha_{i}[X_{w_{0}s_{i}}][X_{w}]   =  \sum_{v\lhd w} [X_{v}].
\]
\end{itemize}
For the details of the proof, see \cite{w-y}. \qed

\bigskip

A Schubert variety $X_{w}(E_\bullet)$ is \textit{defined by inclusions} if it can be
described as the set of all flags $F_{\bullet}$ where $F_{i} \subset
E_{j}$ or $E_{i} \subset F_{j}$ for some collection of pairs $i,j$.

\begin{PAP}
\textit{\!\! 
\emph{(}Gasharov-Reiner \cite{GR2000}\emph{)}
A Schubert variety $X_w$ is defined by inclusions if and only if $w$ avoids 4231, 35142, 42513, 351624.   
}
\end{PAP}

The four patterns appearing in this property have two other
interesting and unexpected connections found using Tenner's Database
of Permutation Pattern Avoidance.

\begin{thm}
\emph{(}Hultman-Linusson-Shareshian-Sj\"ostrand \cite{HLSS}\emph{)}
The number of regions in the inversion arrangement for $w$ is at most
the number of elements below $w$ in Bruhat order.  The two quantities
are equal if and only if $w$ avoids 4231, 35142, 42513, 351624.
\end{thm}

Given a subset $S$ of $[n] \times [n]$, let $mat_q(n,S,r)$ be the
number of $n \times n$ matrices over $\mathbb{F}_q$ with rank $r$,
none of whose nonzero entries lie in $S$.  For example, if $S=\emptyset$, then
$$
mat_q(n,\emptyset,n) = q^{\abechs{n}{2}}(q-1)^n \prod_{i=1}^n (1+q+
\ldots + q^{i-1}) = q^{\abechs{n}{2}}(q-1)^n P_{w_0}(q)
$$
where $w_0$ is the longest element of $S_n$ and $P_{w_0}(q)$ is the
Poincar\'e polynomial for $X_{w_0}$.   

\begin{thm}
\emph{(}Lewis-Morales \cite{lewis.morales.2014}\emph{)}
Fix a permutation $w$ in
  $S_n$, and let $D(w)$ be its permutation diagram. We have that
  $$mat_q(n,D(w),n)/(q-1)^n = q^{n(n-1) - inv(w)} P_{w w_0 } (q^{-1})$$
  If and only if
  $w$ avoids 1324, 24153, 31524, and 426153 (the reverses of the
  patterns in Property 5).  
\end{thm}

The theorem above was originally part of a more general conjecture by
Klein-Lewis-Morales.  We state the part that is still open.

\begin{conj}
  \emph{(}Klein-Lewis-Morales \cite[Conj. 5.1 and Conj
  6.6]{klein.lewis.morales.2013}\emph{)} Using the notation above,
  $mat_q(n, D(w), n)/(q-1)^n$ is a polynomial function of $q$ which
  is coefficient-wise less than or equal to
  $q^{\abechs{n}{2}-inv(w)}P_w(q)$.
\end{conj}

Recently, Albert and Brignall have shown that the enumeration of
Schubert varieties defined by inclusions has a nice generating
function and recurrence relation.  Once again, it is interesting to
compare this formula with \eqref{gf:haiman} and \eqref{gf:factorial}.   

\begin{thm} 
\emph{(}Albert-Brignall  \cite{albert.brignall}\emph{)} 
Let $f(n)$ be the number of permutations in $S_n$ which
  avoid 4231, 35142, 42513, and 351624.  Then, we have the generating
  function
\[
\sum f(n) t^n = \frac{1-3t-2t^2-(1-t-2t^2)\sqrt{1-4t}}{1-3t-(1-t+2t^2)\sqrt{1-4t}}.
\]
\end{thm}

Gasharov-Reiner give a nice description of the cohomology rings of
Schubert varieties defined by inclusions.  This result has been
extended by Reiner-Woo-Yong in a beautiful way which relates to
Fulton's essential set which is a subset of the diagram of a
permutation.  In order to describe it here, let us first recall
Carrell's result on the cohomology of Schubert varieties.
\begin{thm}
\emph{(}Carrell \cite{carrell92}\emph{)} 
$H^{*}(X_{w}) \approx H^{*}(G/B)/I_{w}$ where
$I_{w}$ is generated by all  $[X_v]$ such that $v \not \leq w$.  
\end{thm}

A permutation $x$ is called \textit{Grassmannian} if $x$ has at most 1
descent. Also, $x$ is \textit{bigrassmannian} if both $x$ and $x^{-1}$
are Grassmannian.  We denote by $\mathrm{Des}(x)$ the set of descents
in $x$. In 1992, Akyildiz-Lascoux-Pragacz gave a description of the
ideal $I_w$ which was then further refined by Reiner-Woo-Yong.

\begin{thm}
\emph{(}Akyildiz-Lascoux-Pragacz \cite{ALP}\emph{)} 
  $I_{w}$ is generated by the set of all $[X_v]$ such that
  $v \not \leq w$ and $v$ is Grassmannian.
\end{thm}

Following \cite{RWY}, for a permutation $w$, let $E(w)$ be the set of
permutations which are minimal elements in Bruhat order in the
complement of the interval $[id,w]$.  The set $E(w)$ is called the
\textit{essential set} of $w$.  Clearly, this notion of essential set
generalizes to all Coxeter groups.

\begin{thm}
\emph{(}Lascoux-Sch\"utzenberger and  
Geck-Kim \cite{G-K}\emph{)} 
The elements in $E(w)$ are bigrassmannian.
\end{thm}

\begin{thm}
\emph{(}Reiner-Woo-Yong \cite{RWY}\emph{)} 
There exists a bijection between $E(w)$ and Fulton's essential set
which is defined as the cells in the diagram of the permutation $D(w)$
which have no cell directly to their right or below.   
\end{thm}

\begin{thm}
\emph{(}Reiner-Woo-Yong \cite{RWY}\emph{)}
  $I_{w}$ is generated by the set of all $[X_v]$ such that
  $v \not \leq w$, $v$ is Grassmannian and there exists some
  bigrassmannian $x \in E(w)$ such $x\leq v$ and
  $\mathrm{Des}(x)=\mathrm{Des}(v)$.
\end{thm}

Reiner-Woo-Yong point out that this generating set for $I_w$ is still
not minimal in general.  This leads to some interesting open
questions.

\begin{ques}
Find a minimal set of generators for $I_{w}$ for all $w \in S_n$.  (See \cite{RWY}).
\end{ques}

\begin{ques}
What is the relationship between $E(w)$ and the defining equations for Schubert varieties in other types?
\end{ques}

\bigskip

The next property relates the Bott-Samelson resolution for a singular
Schubert variety and the Kazhdan-Lusztig basis elements to pattern
avoidance.  A resolution of a singular variety is called a
\textit{small resolution} if for every $r>0$, the space of points of
$X$ where the fiber over the point in the resolution has dimension $r$
is of codimension greater than $2r$.  In words, the singular points
where the resolution has to blow up the dimension a lot are rare in a
small resolution.  One reason that people care about small resolutions
is that   the intersection homology of a variety is just the homology
of a small resolution of the variety.  

\begin{PAP}
\textit{\!\!
\emph{(}Deodhar \cite{d}, Billey-Warrington \cite{BW-hexagon}\emph{)}
The following are equivalent.
\begin{enumerate}
\item $C'_{w}=C'_{s_{i_{1}}} C'_{s_{i_{2}}}\cdots C'_{s_{i_{p}}} $ for some (or any) reduced expression
$w=s_{i_{1}}s_{i_{2}}\cdots s_{i_{p}}$.
\item The Bott-Samelson resolution of $X_{w}$ is small.
\item $\displaystyle \sum_{v\leq w} t^{l(v)} P_{v,w}(t) = (1+t)^{l(w)}$.  
\item For each $v\leq w$, the Kazhdan-Lusztig polynomial can be written as
\[
\displaystyle P_{v,w}(t)=\sum_{\sigma \in E(v,w)} t^{\mathrm{defect}(\sigma )}.
\]
\item $w$ is 321-hexagon avoiding, that is, $w$ avoids
\[
321, 56781234, 56718234, 46781235, 46718235 .
\]
\end{enumerate}
}
\end{PAP}
The equivalence of the first four properties was given by Deodhar
\cite{d}.  Showing these properties have a pattern avoidance
characterization is due to Billey-Warrington \cite{BW-hexagon}.
Deodhar's theorem extends to all Weyl groups and in each case there is
again a pattern avoidance characterization due to Billey-Jones
\cite{Billey-Jones}.    

We should explain Deodhar's terminology $\mathrm{defect}(\sigma)$ and
$E(v,w)$ because we believe that they might have important implications for
answering Question~\ref{q:kl.poly}.  First, fix a reduced expression
for $w$.  This corresponds with a string diagram $S$ for $w$.  Think
of each crossing in the string diagram as optional.  Then $E(v,w)$ is
the set of all string diagrams $\sigma$ for $v$ obtained from $S$ by
choosing some subset of the crossings.  The \textit{defect} of
$\sigma$ is the number of times two strings come together that have
previously crossed an odd number of times in the string diagram, as one
progresses vertically.    Thus, 
$$
 P_{v,w}(t)=\sum_{\sigma \in E(v,w)} t^{\mathrm{defect}(\sigma )}
$$
is precisely the sort of combinatorial formula for the Kazhdan-Lusztig
polynomials we would like to have.  Deodhar has shown that for every
pair $v,w \in S_n$ there exists a set of string diagrams $E(v,w)$ for
which the same formula holds.  The only drawback is that in order to
find $E(v,w)$ one must basically compute $P_{v,w}$ using another
method first.

The next pattern property due to Tenner concerns a subset of the
321-hexagon avoiding permutations.

\begin{PAP} 
  \textit{\!\!  
\emph{(}Tenner \cite{t3}\emph{)}
The principal order ideal below $w$ in Bruhat order is
    isomorphic to a Boolean lattice if and only if $w$ is $321$ and $3412$
    avoiding. Equivalently, the Bott-Samelson resolution of $X_w$ is
    isomorphic to $X_w$.  }
\end{PAP}

Thus, a permutation is called \textit{Boolean} if it is $321$ and
$3412$ avoiding.  These permutations give rise to a
familiar enumerative sequence.  

\begin{thm}\emph{(}Fan \cite{fan1}, West \cite{west2}\emph{)}
The number of Boolean permutations in $S_{n}$ is
the Fibonacci number $F_{2n-1}$, e.g. $F_{1}=1,F_{3}=2,F_{5}=5$.
\end{thm}

\bigskip

The next property relates Kazhdan-Lusztig polynomials to a filtration
on permutations.  It was conjectured by Billey-Braden \cite{BiBr} and
proved by Woo \cite{Woo-Billey-Weed}.

\begin{PAP} 
\textit{\!\!
\emph{(}Woo-Billey-Weed \cite{Woo-Billey-Weed}\emph{)} 
The Kazhdan-Lusztig polynomial $P_{id,w}(1)=2$ if and only if $w$ avoids 653421, 632541, 463152, 526413,
546213, and 465132 and the singular locus of $X_{w}$ has exactly 1 component.
}
\end{PAP}

To define a filtration on permutations in a similar way, let's make
the following definition.

\begin{defn} Let 
$\displaystyle KL_{m} = \{w \in S_{\infty } \ | \ P_{id,w}(1)\leq m \}$.
\end{defn}

For example, we know from Property 1 that $KL_{1}$ is the set of
permutations avoiding $3412$ and $4231$.  Similarly Billey-Weed used
Woo's theorem to show that $KL_{2}$ is characterized by the 66 permutation
patterns of length $\leq 8$ below.  This result is in an appendix to \cite{Woo-Billey-Weed}.  
\\
 
 (4 5 1 2 3) (3 4 5 1 2) (5 3 4 1 2) (5 2 3 4 1) (4 5 2 3 1) \vspace{0.05in}

 (3 5 1 6 2 4) (5 2 3 6 1 4) (5 2 6 3 1 4) (6 2 4 1 5 3) (5 2 4 6 1 3) \vspace{0.05in}

 (4 6 2 5 1 3) (5 2 6 4 1 3) (5 4 6 2 1 3) (3 6 1 4 5 2) (4 6 1 3 5 2) \vspace{0.05in}

 (3 6 4 1 5 2) (4 6 3 1 5 2) (5 3 6 1 4 2) (4 6 5 1 3 2) (4 2 6 3 5 1) \vspace{0.05in}

 (6 3 2 5 4 1) (6 3 5 2 4 1) (6 4 2 5 3 1) (6 5 3 4 2 1) \vspace{0.05in}

 (3 6 1 2 7 4 5) (6 2 3 1 7 4 5) (6 2 4 1 7 3 5) (3 4 1 6 7 2 5) \vspace{0.05in}

 (4 2 3 6 7 1 5) (4 2 6 3 7 1 5) (4 2 6 7 3 1 5) (3 7 1 2 5 6 4) \vspace{0.05in}

 (7 2 3 1 5 6 4) (3 7 1 5 2 6 4) (3 7 5 1 2 6 4) (7 5 2 3 1 6 4) \vspace{0.05in}

 (6 2 5 1 7 3 4) (7 2 6 1 4 5 3) (3 4 1 7 5 6 2) (3 5 1 7 4 6 2) \vspace{0.05in}

 (4 5 1 7 3 6 2) (4 2 3 7 5 6 1) (5 3 4 7 2 6 1) (4 2 7 5 6 3 1) \vspace{0.05in}

 (3 4 1 2 7 8 5 6) (4 2 3 1 7 8 5 6) (3 4 1 7 2 8 5 6) \vspace{0.05in}

 (4 2 3 7 1 8 5 6) (4 2 7 3 1 8 5 6) (3 5 1 2 7 8 4 6) \vspace{0.05in}

 (5 2 3 1 7 8 4 6) (5 2 4 1 7 8 3 6) (3 4 1 2 8 6 7 5) \vspace{0.05in}

 (4 2 3 1 8 6 7 5) (3 4 1 8 2 6 7 5) (4 2 3 8 1 6 7 5) \vspace{0.05in}

 (4 2 8 3 1 6 7 5) (3 4 1 8 6 2 7 5) (4 2 3 8 6 1 7 5) \vspace{0.05in}

 (4 2 8 6 3 1 7 5) (3 5 1 2 8 6 7 4) (5 2 3 1 8 6 7 4) \vspace{0.05in}

 (3 6 1 2 8 5 7 4) (6 2 3 1 8 5 7 4) (5 2 4 1 8 6 7 3) \vspace{0.05in}

 (6 2 5 1 8 4 7 3)
\\

\bigskip

A local ring $R$ is a \textit{local complete intersection (LCI)}
  if it is the quotient of some regular local ring by an ideal
  generated by a regular sequence. A variety is \textit{LCI} if every
  local ring is LCI.

\begin{PAP}
\textit{\!\!
\emph{(}\'Ulfarsson-Woo \cite{ulfarsson.woo.2013}\emph{)} 
A Schubert variety $X_w$ is LCI
if and only if w avoids 53241, 52341, 52431, 35142, 42513, and 426153.
}
\end{PAP}

Since regular local rings are LCI, smooth varieties are automatically
LCI. Furthermore, LCI varieties are Gorenstein and hence
Cohen-Macaulay. Thus, being LCI can be viewed as saying that the
singularities are in some sense mild. Compare the above criterion with
Property 1 (for smoothness) and Property 4 (for Gorenstein property).

\bigskip A permutation is \textit{vexillary} if it avoids $2143$,
introduced by Lascoux-Sch\"utzenberger in 1982
\cite{lascouxschutzenbergerschubert}.  The word vexillary is related
to flags, hence the choice.  We say $w$ is \textit{covexillary} if $w$ avoids $3412$.
There are so many interesting things to say related to vexillary and
covexillary permutations so the tenth property has 3 parts.

\begin{finalfirstPAP}
\textit{\!\!
\begin{enumerate}
\item \emph{(}Edelman-Greene \cite{EG}\emph{)}
  The number of reduced words for a vexillary permutation $v$ is
  equal to the number of standard tableaux of shape determined by
  sorting the lengths of the rows of the diagram of $v$.
\item \emph{(}Edelman-Greene \cite{EG}\emph{)}
  The Stanley symmetric function $F_{v}$ is a Schur function if
  and only if $v$ is vexillary.  Here 
\[
F_{v} = \sum_{\mathbf{a}=a_{1}a_{2}\dots a_{k} \in R(v)} \hspace{.2in}
\sum_{i_{1}\leq \dots \leq i_{k} \in C(\mathbf{a})} x_{i_{1}}
x_{i_{2}}\cdots x_{i_{k}}
\]
where $R(v)$ are the reduced words for $v$ and $C(\mathbf{a})$ are the
weakly increasing sequences of positive integers such that $i_{j}<
i_{j+1}$ if $a_{j}<a_{j+1}$.
\item \emph{(}Tenner \cite{t3}\emph{)}
The permutation $v$ is vexillary if and only if for every permutation
$w$ containing $v$, there exists a reduced decomposition $\mathbf{a}
\in R(w)$ containing a shift of some $\mathbf{b} \in R(v)$ as a factor.  
\end{enumerate}
}
\end{finalfirstPAP}

The next is a list of properties of vexillary permutations related to geometry of Schubert varieties.

\begin{finalsecondPAP}
\textit{\!\!
\begin{enumerate}
\item \emph{(}Fulton \cite{Fulton1}\emph{)} 
  Recall, Fulton's \textit{essential set} for $w$ is the collection of
  cells in the diagram of $w$ with no neighbor directly east or south.
  If $w$ is vexillary, these cells lie on an increasing piecewise
  linear curve.
\item \emph{(}Lascoux \cite{Lascoux95}\emph{)}
There exists a combinatorial approach to computing the Kazhdan-Lusztig
polynomials $P_{v,w}$ when $w$ is \textit{covexillary}.   
\item \emph{(}Li-Yong \cite{L-Y}\emph{)}
  There exists a combinatorial rule for computing multiplicities for
  $X_w$ when $w$ is covexillary.   
\end{enumerate}
}
\end{finalsecondPAP}

We say a permutation $w$ is \textit{$k$-vexillary} if its Stanley
symmetric function $F_{w}$ has at most $k$ terms of Schur functions in
its expansion.  For example, $F_{2143} = s_{(2)} + s_{(1,1)}$, so
$2143$ is 2-vexillary.   

\begin{finalthirdPAP}
\textit{\!\!
\emph{(}Billey-Pawlowski \cite{B-Paw}\emph{)} 
The $k$-vexillary permutations are characterized by a finite set of
patterns for all $k$.  
}
\end{finalthirdPAP}

For example, if $w$ is a permutation, then the following hold.  
\begin{enumerate}
\item $w$ is 2-vexillary if and only if $w$
avoids 35 patterns in $S_{5}, S_{6}, S_{7}, S_{8}$.   
\item $w$ is 3-vexillary if and only if $w$
avoids 91 patterns in $S_{5}, S_{6}, S_{7}, S_{8}$.
\end{enumerate}

\bigskip

The list of 2-vexillary patterns is given as follows: \\
\begin{small}
  \ \ \ (3 2 1 5 4) (2 1 5 4 3) 
  (2 1 4 3 6 5) (2 4 1 3 6 5) (3 1 4 2 6 5)  (3 1 2 6 4 5) 
  (2 1 4 6 3 5) (2 4 1 6 3 5) (2 3 1 5 6 4) (2 1 5 3 6 4)
  (3 1 5 2 6 4) (4 2 6 1 5 3) (5 2 7 1 4 3 6) (5 1 7 3 2 6 4)
  (4 2 6 5 1 7 3) (2 5 4 7 1 6 3) (5 4 7 2 1 6 3) (5 2 7 6 1 4 3)
  (6 1 8 3 2 5 4 7) (2 6 4 8 1 5 3 7) (6 4 8 2 1 5 3 7) (2 6 5 8 1 4 3
  7) (6 5 8 2 1 4 3 7) (5 1 7 3 6 2 8 4) (5 1 7 6 3 2 8 4) (6 1 8 3 7
  2 5 4) (6 1 8 7 3 2 5 4) (2 5 4 7 6 1 8 3) (5 4 7 2 6 1 8 3) (5 4 7
  6 2 1 8 3) (2 6 4 8 7 1 5 3) (6 4 8 7 2 1 5 3) (2 6 5 8 7 1 4 3) (6
  5 8 2 7 1 4 3) (6 5 8 7 2 1 4 3).   
\end{small}

We have given 10+ properties of Schubert varieties which are amenable
to pattern avoidance in their characterization.  This is just the
beginning of all the consequences for the Lakshmibai-Sandhya Theorem.
In the next section, we will discuss how pattern avoidance extends to
other Lie types and Coxeter groups.  

There are two further directions/consequences concerning special
families of varieties we should note.  First is the $GL_p\times
GL_q$-orbits in the flag variety for $GL_{p+q}$ with rationally smooth
closure.  These varieties are special cases of the symmetric varieties
studied by Springer.  McGovern has characterized which symmetric
varieties in this case are rationally smooth by using patterns
involving a multiset of numbers and $+$ and $-$ signs.  See
\cite{mcgovern.2009} for further details.  Similar results are given
in type $C$ by McGovern and Trapa \cite{mcgovern.trapa}.

Second, pattern avoidance also comes up in the study of Peterson
varieties.  The Peterson variety for $\mathbb{C}^n$ is the collection
of complete flags $F_\bullet$ such that $N\cdot F_i \subset F_{i+1}$
for all $1\leq i <n$ where $N$ is a fixed regular nilpotent matrix.
Up to isomorphism, the variety is independent of the choice of $N$.
Insko and Yong gave a combinatorial description of the singular locus
of the Peterson variety which involves the patterns 123 and 2143 among other
conditions \cite{Insko.Yong}.

\sectionspace
\section{Pattern avoidance for Coxeter groups}
In this section, we study pattern avoidance properties for Coxeter groups. First, we recall the definition of Coxeter groups and their basic properties. For details, see \cite{b-b,Hum}.

\subsection{A quick review on Coxeter groups} 
A \textit{Coxeter graph} is a simple graph with vertices $\{1,2,\dots , n\}$ and edges labeled by $\Z_{\geq 3}\cup \infty$. 
The Coxeter group associated to a Coxeter graph $G$ is the group generated by $S=\{s_{1}, s_{2},\dots, s_{n} \}$ with relations
\begin{enumerate}
\item $s_{i}^{2}=1.$
\item $s_{i}s_{j} = s_{j} s_{i}$ \hspace{.01in} if $i,j$ not adjacent in $G$. 
\item $\underbrace{s_{i}s_{j}s_{i} \cdots}_{m(i,j)\text{ generators}} =
\underbrace{s_{j} s_{i} s_{j} \cdots}_{m(i,j)\text{ generators}}$ if $i,j$ connected by edge labeled $m(i,j)<\infty$.
\end{enumerate}

\vspace{0.05in} 

Since a Coxeter group is completely determined by its Coxeter graph,
we simply need to draw the graph to refer to the associated Coxeter
group.  Conventionally, we drop the label 3 from any edge in pictures for
simplicity.
\begin{center}
\hspace{0in}
  \xymatrix @-1pc {
\gn_1 \ar@{-}[r]^{4} & \gn_2 \ar@{-}[r]^{3} & \gn_3 \ar@{-}[r]^{3} & \gn_4\\
} \hspace{.1in}  $\approx$ \hspace{.1in}
  \xymatrix @-1pc {
\gn_1 \ar@{-}[r]^{4} & \gn_2 \ar@{-}[r] & \gn_3 \ar@{-}[r] & \gn_4\\
} 
\end{center}

\bigskip

\begin{ex}
The following are examples of Coxeter groups.

\begin{itemize}
\item[(1)] Dihedral groups:  $\mathrm{Dih}_{10}$ is   \hspace{.3in}   \xymatrix @-1pc {
\gn_1 \ar@{-}[r]^{5} & \gn_2 }

\item[(2)] Symmetric groups:   $S_{5}$ is \hspace{.3in} \xymatrix @-1pc {
\gn_1 \ar@{-}[r] & \gn_2 \ar@{-}[r] & \gn_3 \ar@{-}[r] & \gn_4
} 

\item[(3)] Hyperoctahedral groups:  $B_{4}$ is \hspace{.3in}    \xymatrix @-1pc {
\gn_1 \ar@{-}[r]^{4} & \gn_2 \ar@{-}[r] & \gn_3 \ar@{-}[r] & \gn_4\\
}

\item[(4)] 
The exceptional Weyl groups: $E_{8}$ is
\[
\unitlength 0.1in
\begin{picture}(  8.4000,  3.2000)( 12.0000,-15.5000)
\put(12.0000,-14.0000){\makebox(0,0)[lb]{\xymatrix @-1pc { \gn_1 \ar@{-}[r] & \gn_2 \ar@{-}[r] & \gn_3 \ar@{-}[r] & \gn_4 \ar@{-}[r] & \gn_5 \ar@{-}[r] & \gn_6 \ar@{-}[r] &  \gn_7 }}}%
\put(20.1000,-16.9000){\makebox(0,0)[lb]{$\gn_{8}$}}%
%
\special{pn 8}%
\special{pa 2040 1550}%
\special{pa 2040 1390}%
\special{fp}%
\end{picture}%
\]
\end{itemize}
\vspace{0.5cm}
Curiously, the exceptional Weyl group $E_{8}$ appears in string theory
and in chemistry related to the symmetry group of the $C_{60}$
molecule and buckyballs \cite{CKS}.
\end{ex}

Fix a Coxeter group $W$ with Coxeter graph $G$.  
The set of \textit{reflections} $R \subset W$ is the set of  all
conjugates of the generators, 
\[
R=\bigcup_{w\in W} w S w^{-1}.
\]
A \textit{reduced expression} of an element $w\in W$ is an expression $w=s_{i_1}\cdots s_{i_k}$ as a product of generators in which $k$ is the minimum among such expressions.
The \textit{length} of $w\in W$ is the length of a reduced expression
for $w$, denoted $\ell(w)$ again.
Bruhat order on the Coxeter group $W$ is the transitive closure of the following relation
\[
x \leq y \quad \text{ if } \quad \ell(x) < \ell(y) \text{ and } xy^{-1} \in R.
\]
It was observed by Chevalley that $x \leq y$ if and only if for any
reduced expression $y= s_{i_{1}} s_{i_{2}} \dots s_{i_{p}}$ there
exists a subexpression which is a reduced expression for $x$, in symbols
$x =s_{i_{1}}^{\sigma_{1}}s_{i_{2}}^{\sigma_{2}} \dots
s_{i_{p}}^{\sigma_{p}} $ for some \textit{mask} $\sigma_{1}\dots
\sigma_{p}\in \{0,1\}^{p}$ \cite{CHEV}.

\bigskip

There are many expressions for any $w \in W$ as a product of
generators, but it is a well known hard problem to tell when two
expressions are equal in a group using only generators and relations.
Luckily, there is an algorithm of finding a canonical representative
for each element of $W$, called the \textit{Mozes numbers game}. See Mozes
1990, Eriksson-Eriksson 1998, Bj\"orner-Brenti \cite{b-b,EE-98,Mozes}. Let us
briefly explain this game/algorithm here.

Replace each edge $(i,j)$ of $G$ by two opposing directed edges labeled $f_{ij}>0$ (for the edge $i\rightarrow j$) and $f_{ji}>0$ (for the edge $j\rightarrow i$) so that $f_{ij}f_{ji} = 4 \mathrm{cos}^{2} \left(\frac{\pi}{ m(i,j)} \right)$ or $f_{ij}f_{ji} =4$ if $m(i,j)=\infty$. 
These labels are fixed in the game once chosen. The following is a useful choice since the labels are all integers. 
\[
\begin{array}{c|c|c}
m(i,j) &   f_{ij} & f_{ji}\\
\hline
3 &	1 &	1\\
4 &	 2 &	1 \\
6 &	3 &	1 \\
\infty &	4 &	1
\end{array}
 \]

 Assume that we are given an element $w=s_{i_{1}}s_{i_{2}}\dots
 s_{i_{p}}\in W$. The canonical presentation of $w$ is obtained as
 follows.  We first assign value $1$ to each vertex $G$.  Next,
 \textit{fire the vertex} $s_{i_1}$.  Here, firing the vertex $s_{i}$ is
 an operation done by adding to the value of each neighbor vertex $j$,
 the current value at the vertex $i$ multiplied by $f_{ij}$, and then
 negating the sign of the value of the vertex $i$.  We continue to fire
 the vertices $s_{i_2}, s_{i_3},\ldots, s_{i_p}$ consecutively.  The
 resulting assignment of values for vertices of $G$, denoted by $G(w)$,
 provides a canonical presentation of the given $w$.  In fact, this
 algorithm satisfies the following properties:

\begin{enumerate}
\item $G(w)$ only depends on the product $s_{i_{1}}s_{i_{2}}\dots
s_{i_{p}}$ and not on the particular choice of expression.

\item The vertex $i$ is negative in $G(w)$ if and only f $ws_{i} <w$.

\item The vertex $i$ never has value 0.
\end{enumerate}

Note, the map $G$ is injective but not surjective on the set of all
integer assignments to the nodes of the Coxeter graph.

\begin{rem}
For $I \subset S$, it is possible to modify the game to get
representatives for $W/W_{I}$ by starting with initial value 0 on
vertices in $I$ and 1's elsewhere.  
Then $ws_{i}=w$ if and only if the vertex $i$ has value 0 in $W/W_{I}$.  
This is useful for Schubert geometry of Grassmannians and affine Grassmannians.
\end{rem}


For a Coxeter group $W$, we can associate to it its \textit{root
  system} $\Phi\subset V= \mathbb{R}^{|S|}$ where $\{\alpha_{s} \colon s
\in S \}$ forms a basis of $V$ \cite[Section 5.4]{Hum}. $W$ acts
linearly on $V$ and $\Phi$ is $W$-invariant.  We denote by $\Phi_{+}$
and $\Phi_{-}$ the set of \textit{positive roots} and the set of
\textit{negative roots}, respectively:
\begin{align*}
&\Phi_{+} = \{\alpha \in \Phi \colon\alpha = \sum c_{s} \alpha_{s} , \
c_{s}\geq 0, \forall s \in S\}, \\
&\Phi_{-} = \{\alpha \in \Phi \colon \alpha = \sum c_{s} \alpha_{s} , \
c_{s} \leq 0\, \forall s \in S\}.
\end{align*}
It follows that $\Phi = \Phi_{+} \cup \Phi_{-}$ (disjoint
union). There is a natural bijection between $R$ and $\Phi_{+}$ which
we will denote by $r \rightarrow \alpha_r$.  Then,
for $r \in R, w \in W$, we have
\[
wr > w  \quad \text{ if and only if } \quad w\alpha_{r} \in \Phi_{+}.
\] 


\begin{ex}
Let $e_{1},\dots , e_{n}$ be the standard orthonormal basis of $\mathbb{R}^{n}$. Then the root system of the Weyl groups of classical types are determined by the following description of $\Phi_{+}$.

\begin{itemize}
\item[]  $A_{n-1}$ : \hspace{.2in} $\Phi_{+} = \{e_{i}-e_{j}\colon 1\leq
  i<j \leq n\}$  
\\
\item[] $B_{n}$  : \hspace{.3in} $\Phi_{+} = \{e_{i}-e_{j}\colon 1\leq i<j
  \leq n \} \cup \{e_{i}+e_{j}\colon 1 \leq i<j \leq n\} \cup \{e_{i}\colon
  1\leq i \leq n \}$  
\\
\item[] $C_{n}$ : \hspace{.3in} $\Phi_{+} = \{e_{i}-e_{j}\colon 1\leq i<j
  \leq n \} \cup \{e_{i}+e_{j}\colon 1 \leq i<j \leq n \} \cup \{2e_{i}\colon 1
  \leq i \leq n \}$  
\\
\item[] $D_{n}$ : \hspace{.3in} $\Phi_{+} = \{e_{i}-e_{j}\colon  1 \leq i<j
  \leq n\} \cup \{e_{i}+e_{j}\colon 1 \leq i<j  \leq n\}$  
\end{itemize}
\end{ex}

The \textit{inversion set} of $w\in W$ is defined to be $w\Phi_{+}
\cap \Phi_{-} $.  In type $A_{n-1}$, these roots are in bijection with
the inversion set of $w$ defined originally.    For a linear function
$H\colon V \longrightarrow \mathbb{R}$, we let
\[
\Pi_{H} = \{\alpha \in  \Phi \colon  H(\alpha ) >0\}.
\]
This is an intersection of the set of roots with a half space.
We say $H$ is \textit{generic} if $H(\alpha ) \neq 0$ for all $\alpha \in \Phi$.

\begin{ex}
If $H_{1}\colon V \longrightarrow \mathbb{R}$ is defined by
$H_{1}(\alpha_{s})=1$ for all $s\in S$, then  $\Pi_{H_{1}}=\Phi_{+}$.  
\end{ex}

\begin{defn}
For each $w\in W$, set $H_{w} = H_{1}\circ w^{-1}$.
Then, we have $\Pi_{H_{w}} = w \Phi_{+}$.
\end{defn}

A key fact is that, if $H$ is generic, then $\Pi_{H} = w\Phi_{+}$ for
some unique $w \in W$. That is, every generic half space determines a
unique $w\in W$ whose inversion set is exactly the negative roots in the given half space. 
Below are the positive roots for two types of Coxeter groups  drawn projectively in 2 dimensions. 
We denote by $\beta_{ij}=e_i-e_j$ for $A_3$.
\vspace{0.05in}
\[
\unitlength 0.1in
\begin{picture}( 43.5000, 18.3000)( 13.5000,-21.9000)
%
\special{pn 8}%
\special{pa 1600 1000}%
\special{pa 3100 1000}%
\special{fp}%
%
\special{pn 8}%
\special{sh 1.000}%
\special{ar 2100 1000 26 26  0.0000000 6.2831853}%
%
\special{pn 8}%
\special{sh 1.000}%
\special{ar 1590 1000 26 26  0.0000000 6.2831853}%
%
\special{pn 8}%
\special{sh 1.000}%
\special{ar 3090 1000 26 26  0.0000000 6.2831853}%
%
\special{pn 8}%
\special{sh 1.000}%
\special{ar 2590 1000 26 26  0.0000000 6.2831853}%
\put(13.5000,-10.7000){\makebox(0,0)[lb]{$\beta_1$}}%
\put(31.5000,-10.7000){\makebox(0,0)[lb]{$\beta_2$}}%
\put(18.5000,-9.2000){\makebox(0,0)[lb]{$\beta_1+\beta_2$}}%
\put(24.1000,-9.2000){\makebox(0,0)[lb]{$\beta_1+2\beta_2$}}%
\put(15.2000,-15.3000){\makebox(0,0)[lb]{$\beta_1$ and $\beta_1+2\beta_2$ are long roots.}}%
\put(15.2000,-17.3000){\makebox(0,0)[lb]{$\beta_1+\beta_2$ and $\beta_2$ are short roots.}}%
\put(15.2000,-5.3000){\makebox(0,0)[lb]{$B_2:$}}%
\put(39.2000,-5.3000){\makebox(0,0)[lb]{$A_3=S_4:$}}%
\put(43.2000,-19.6000){\makebox(0,0)[lb]{$\beta_{12}=\beta_1, \ \beta_{23}=\beta_2, \ \beta_{34}=\beta_3, $}}%
\put(43.2000,-21.6000){\makebox(0,0)[lb]{$\beta_{13}=\beta_1+\beta_2, \ \beta_{24}=\beta_2+\beta_3, $}}%
\put(46.4000,-23.6000){\makebox(0,0)[lb]{$\beta_{14}=\beta_1+\beta_2+\beta_3.$}}%
%
\special{pn 8}%
\special{pa 5200 618}%
\special{pa 4750 1514}%
\special{fp}%
%
\special{pn 8}%
\special{pa 5200 618}%
\special{pa 5646 1514}%
\special{fp}%
%
\special{pn 8}%
\special{pa 5646 1514}%
\special{pa 4974 1052}%
\special{fp}%
%
\special{pn 8}%
\special{pa 4750 1514}%
\special{pa 5416 1052}%
\special{fp}%
\put(45.1000,-16.3200){\makebox(0,0)[lb]{$\beta_{12}$}}%
\put(57.0000,-16.3200){\makebox(0,0)[lb]{$\beta_{34}$}}%
\put(47.2000,-10.8200){\makebox(0,0)[lb]{$\beta_{13}$}}%
\put(51.3000,-14.2000){\makebox(0,0)[lb]{$\beta_{14}$}}%
\put(55.1000,-10.8200){\makebox(0,0)[lb]{$\beta_{24}$}}%
\put(51.3500,-5.4000){\makebox(0,0)[lb]{$\beta_{23}$}}%
%
\special{pn 8}%
\special{sh 1.000}%
\special{ar 5200 610 26 26  0.0000000 6.2831853}%
%
\special{pn 8}%
\special{sh 1.000}%
\special{ar 5200 1210 26 26  0.0000000 6.2831853}%
%
\special{pn 8}%
\special{sh 1.000}%
\special{ar 5420 1050 26 26  0.0000000 6.2831853}%
%
\special{pn 8}%
\special{sh 1.000}%
\special{ar 4980 1060 26 26  0.0000000 6.2831853}%
%
\special{pn 8}%
\special{sh 1.000}%
\special{ar 4760 1500 26 26  0.0000000 6.2831853}%
%
\special{pn 8}%
\special{sh 1.000}%
\special{ar 5630 1500 26 26  0.0000000 6.2831853}%
\end{picture}%
\]
\vspace{0.05in}
%

For example, if a given half space $\Pi_H$ contains all the positive
roots $\beta_{ij}$'s except for $\beta_{23}$ and $\beta_{24}$, then
$\Pi_H = \Pi_{H_w}$ for  $w=2431$.

\bigskip

\subsection{Coxeter patterns}

Each subset $I \subset S$ generates a subgroup $W_I$.   A subgroup 
$W'\subset W$ which is conjugate to $W_I$ for some $I$ is called
a {\em parabolic} subgroup.  The $W_I$'s themselves are known as
{\em standard} parabolic subgroups.

A parabolic subgroup $W' = xW_Ix^{-1}$ of $W$ is again a Coxeter
group, with simple reflections $S' = xIx^{-1}$ and reflections $R' = R
\cap W'$.  Note that $S' \not\subset S$ unless $W'$ is standard.  

We denote the length function and the Bruhat-Chevalley order for
$(W',S')$ by $l'$ and $\le'$, respectively.  If $W' = W_I$ then
\[l' = l|_{W'}\;\text{and}\; \mathord{\le'} = \mathord{\le}|_{W'\times W'},\]
but in general we only have $l'(w) \le l(w)$ and $x \le' y \implies x \le y$.
For instance, if $W' \subset \mfS_4$ is generated by the reflections
$r_{23} = 1324$ and $r_{14} = 4231$, then $r_{23} \le r_{14}$ although they are
not comparable for $\le'$.

The following theorem/definition generalizes the flattening function
for permutations to all Coxeter groups.  The following theorem is
closely related to a theorem due to Dyer on reflection subgroups
\cite[Thm. 1.4]{dyer.1991}.

\begin{thm} \label{coset theorem}\cite{BiBr} Let $W' \subset W$ be a parabolic subgroup.
There is a unique function $\fl\colon W\to W'$, the 
{\em pattern map} for $W'$, satisfying the following two properties.
\begin{enumerate} \item[(a)] 
The map $\fl$ is $W'$-equivariant: $\fl(wx) = w\fl(x)$ for all $w\in W'$, $x\in W$.
\item[(b)] If $\fl(x) \le' \fl(wx)$ for some $w\in W'$, 
then $x \le wx$.
\end{enumerate}
In particular, $\fl$ restricts to the identity map on $W'$.
\end{thm}
If $W' = W_I$ is a standard parabolic, then (b) can be strengthened to
``if and only if''.  In this case the result is well-known.

To show uniqueness, note that (a) implies that $\fl$ is determined by
the set $\fl^{-1}(1)$, and (b) implies that $\fl^{-1}(1) \cap W'x$ is
the unique minimal element in $W'x$.  Existence is more subtle; it is
not immediately obvious that the function so defined satisfies (b).
We give a construction of a function $\fl$ that satisfies (a) and
(b).

Recall $V$ is the real vector space spanned by the roots in the root
system $\Phi$ associated to the Coxeter group $W$.   If $U \subset V$
is a linear subspace, then we use the following notations:
\begin{itemize}
\item[] $\Phi^{U}:=\Phi \cap U $, a \textit{root subsystem} of $\Phi$,
\item[] $W^{U}$ is the group generated by reflections $r_{\alpha}$ for $\alpha \in \Phi^{U}$,
\item[] $R^{U} := R \cap W^{U}$.
\end{itemize}

One can show that $W^{U}$ is a parabolic subgroup of $W$ assuming $W$ is finite, see
\cite[\S1.12]{Hum}.  Note that not all subgroups of $W$ generated by
reflections are parabolic subgroups.  For example, for $B_{3}$, the
group generated by reflections over the $e_{i}$'s is not parabolic.

By the uniqueness statement in Theorem~\ref{coset theorem}, we can use
the sets $\Pi_H$ defined earlier to realize $\fl\colon W\longrightarrow
W^U$.  In fact, $\fl(w)$ is the unique element $x \in W^{U}$ such that
\begin{align*}
w\Phi_{+} \cap U &= \{\alpha \in U \cap  \Phi \colon H_{w} (\alpha )>0 \}\\
	&= \{\alpha \in   \Phi^{U} \colon H' (\alpha )>0 \} \text{ where } H' = H_{w}|_{U}\\
        &= x \Phi_{+}^{U}.
\end{align*}
This realization of the flattening map for Weyl groups was first given
by Billey-Postnikov \cite{BP-smooth} even though it was published
later than \cite{BiBr}.   The delay is explained below.

\begin{ex}
Let $U=\text{span}\langle \beta_{23}, \beta_{34} \rangle$. Then $\fl_U(2431)=243$. See the picture in the previous example.
\end{ex}

\subsection{Applications of Coxeter Patterns}
Let us denote 
\begin{itemize}
\item[] $G$ : a semisimple simply-connected complex Lie group,
\item[] $B\subset G$ : a Borel subgroup, 
\item[] $T \subset B$ : a maximal torus, 
\item[] $W = N(T)/T$ : the Weyl group (a finite Coxeter group),
\item[] $\Phi \subset V$ : the  associated root system
\end{itemize}
where $N(T)$ is the normalizer of $T$ in $G$.  The finite Weyl groups
(or root systems) that arise this way have been completely classified
into types $A_{n}, B_{n}, C_{n}, D_{n}, E_{6},E_{7}, E_{8}, F_{4},
G_{2}$.  The \textit{Bruhat decomposition} enables us to partition $G$
using the Borel subgroup and the Weyl group:
\[
\displaystyle G = \bigcup_{w \in W} BwB.
\]
The quotient $G/B$ is called the \textit{(generalized) flag manifold}, and Schubert cells and Schubert varieties of $G/B$ are 
\[
C_{w} = B\cdot w, \quad X_{w} = \overline{B\cdot w} 
\]
for each $w\in W$, respectively.  

The next theorem characterizes all smooth Schubert varieties for any
semisimple simply-connected complex Lie group $G$.  To state the
theorem, we need a few more definitions.

\begin{defn}
  A Coxeter group $W$ is \textit{stellar} if its Coxeter graph has one
  central vertex and all other vertices are only adjacent to it.
\end{defn}

The stellar Coxeter groups corresponding to the Weyl groups of types
$A, B, C, D, E, F$ and $G$ (except for $A_2$) are drawn below where a
double edge and a triple edges mean that the label of the
corresponding edge is 4 and 6, respectively.  Note, that the Weyl
groups of types $B_n$ and $C_n$ are isomorphic, but the pattern map
works slightly differently on each so we list their Dynkin diagram
instead of their Coxeter graph.

\[
\unitlength 0.1in
\begin{picture}( 43.3700, 11.3000)( 12.0000,-18.9000)
\put(12.0000,-10.6000){\makebox(0,0)[lb]{$B_2=$}}%
%
\special{pn 8}%
\special{pa 2800 998}%
\special{pa 3800 998}%
\special{fp}%
%
\special{pn 8}%
\special{pa 1600 1570}%
\special{pa 2100 1570}%
\special{fp}%
%
\special{pn 8}%
\special{sh 1.000}%
\special{ar 2110 1600 38 38  0.0000000 6.2831853}%
%
\special{pn 8}%
\special{pa 1600 1630}%
\special{pa 2100 1630}%
\special{fp}%
%
\special{pn 8}%
\special{pa 1600 1600}%
\special{pa 2100 1600}%
\special{fp}%
%
\special{pn 8}%
\special{sh 1.000}%
\special{ar 1600 1600 38 38  0.0000000 6.2831853}%
%
\special{pn 8}%
\special{pa 4500 998}%
\special{pa 5000 998}%
\special{fp}%
%
\special{pn 8}%
\special{sh 1.000}%
\special{ar 2110 1000 38 38  0.0000000 6.2831853}%
%
\special{pn 8}%
\special{sh 1.000}%
\special{ar 2800 998 38 38  0.0000000 6.2831853}%
%
\special{pn 8}%
\special{sh 1.000}%
\special{ar 1600 1000 38 38  0.0000000 6.2831853}%
%
\special{pn 8}%
\special{sh 1.000}%
\special{ar 3800 998 38 38  0.0000000 6.2831853}%
%
\special{pn 8}%
\special{sh 1.000}%
\special{ar 3300 998 38 38  0.0000000 6.2831853}%
%
\special{pn 8}%
\special{sh 1.000}%
\special{ar 4500 998 38 38  0.0000000 6.2831853}%
%
\special{pn 8}%
\special{sh 1.000}%
\special{ar 5000 998 38 38  0.0000000 6.2831853}%
%
\special{pn 8}%
\special{pa 5400 798}%
\special{pa 5000 998}%
\special{fp}%
%
\special{pn 8}%
\special{sh 1.000}%
\special{ar 5400 798 38 38  0.0000000 6.2831853}%
%
\special{pn 8}%
\special{sh 1.000}%
\special{ar 2800 1598 38 38  0.0000000 6.2831853}%
%
\special{pn 8}%
\special{sh 1.000}%
\special{ar 3800 1598 38 38  0.0000000 6.2831853}%
%
\special{pn 8}%
\special{sh 1.000}%
\special{ar 3300 1598 38 38  0.0000000 6.2831853}%
%
\special{pn 8}%
\special{sh 1.000}%
\special{ar 4500 1598 38 38  0.0000000 6.2831853}%
%
\special{pn 8}%
\special{sh 1.000}%
\special{ar 5500 1598 38 38  0.0000000 6.2831853}%
%
\special{pn 8}%
\special{sh 1.000}%
\special{ar 5000 1598 38 38  0.0000000 6.2831853}%
%
\special{pn 8}%
\special{pa 1600 980}%
\special{pa 2100 980}%
\special{fp}%
%
\special{pn 8}%
\special{pa 1600 1020}%
\special{pa 2100 1020}%
\special{fp}%
%
\special{pn 8}%
\special{pa 4500 1598}%
\special{pa 5000 1598}%
\special{fp}%
%
\special{pn 8}%
\special{pa 5000 1578}%
\special{pa 5500 1578}%
\special{fp}%
%
\special{pn 8}%
\special{pa 5000 1618}%
\special{pa 5500 1618}%
\special{fp}%
%
\special{pn 8}%
\special{pa 2800 1598}%
\special{pa 3300 1598}%
\special{fp}%
%
\special{pn 8}%
\special{pa 3300 1578}%
\special{pa 3800 1578}%
\special{fp}%
%
\special{pn 8}%
\special{pa 3300 1618}%
\special{pa 3800 1618}%
\special{fp}%
%
\special{pn 8}%
\special{sh 1.000}%
\special{ar 5400 1198 38 38  0.0000000 6.2831853}%
%
\special{pn 8}%
\special{pa 5400 1198}%
\special{pa 5000 998}%
\special{fp}%
\put(24.0000,-10.5700){\makebox(0,0)[lb]{$A_3=$}}%
\put(12.0000,-16.6000){\makebox(0,0)[lb]{$G_2=$}}%
\put(24.0000,-16.6000){\makebox(0,0)[lb]{$B_3=$}}%
\put(41.0000,-16.6000){\makebox(0,0)[lb]{$C_3=$}}%
\put(41.0000,-10.6000){\makebox(0,0)[lb]{$D_4=$}}%
%
\special{pn 8}%
\special{pa 5030 1600}%
\special{pa 5110 1540}%
\special{fp}%
%
\special{pn 8}%
\special{pa 1630 1600}%
\special{pa 1710 1540}%
\special{fp}%
%
\special{pn 8}%
\special{pa 1630 1600}%
\special{pa 1710 1660}%
\special{fp}%
%
\special{pn 8}%
\special{pa 5030 1600}%
\special{pa 5110 1660}%
\special{fp}%
%
\special{pn 8}%
\special{pa 2090 1000}%
\special{pa 2010 940}%
\special{fp}%
%
\special{pn 8}%
\special{pa 2090 1000}%
\special{pa 2010 1060}%
\special{fp}%
%
\special{pn 8}%
\special{pa 3790 1600}%
\special{pa 3710 1540}%
\special{fp}%
%
\special{pn 8}%
\special{pa 3790 1600}%
\special{pa 3710 1660}%
\special{fp}%
\put(21.4000,-20.6000){\makebox(0,0)[lb]{Dynkin diagrams of stellar root systems}}%
\end{picture}%
\]
\vspace{0.1in}

%
%
%
%
%
%
%
%

\begin{thm}\label{BP sm for Schubert}
  \emph{(}Billey-Postnikov \cite{BP-smooth}\emph{)} A Schubert variety $X_{w}$ is smooth if and only if
  for every stellar parabolic subgroup $W^{U}$, the Schubert variety
  $X_v$ for $v=\fl_U(w)$ is smooth in $G^U/B^U$.
\end{thm}

Here $G^U$ is a semisimple Lie group with Weyl group $W^U$ and $B^U$
is one of its Borel subgroups.  We remark that $W^{U}$ might not be
the same type as $W$. For example, for Weyl groups of type $C$ and $D$
there will appear $W^{U}$ of type $A$.  In fact, the type $A$ singular
patterns are most common.  If the Coxeter graph of $W$ has only edges
labeled by 3's, we say $W$ is \textit{simply laced}.  If $W$ is simply
laced, then all of its parabolic subgroups are also simply laced.

It turns out that there are very few patterns for which the
corresponding Schubert varieties are singular among stellar reduced,
irreducible Weyl groups; 2 patterns in $A_{3}$, 1 pattern in $B_{2}$,
6 patterns in $B_{3}$ and $C_{3}$, 1 pattern in $D_{4}$, 5 patterns in
$G_{2}$.  Note that all Schubert varieties of type $A_2$ are smooth.

\begin{ex}
In type $B_{n}$ using the classical pattern avoidance on signed permutations, the smooth Schubert varieties are classified by avoiding the following 17 patterns
\begin{align*}
   &(-2, -1), \\
   &(1, 2, -3), \ (1, -2, -3), \ (-1, 2, -3), \ (2, -1, -3), \ (-2, 1, -3), \ (3, -2, 1) \\
   &(2, -4, 3, 1), \ (-2, -4, 3, 1), \ (3, 4, 1, 2), \ (3, 4, -1, 2), \ (-3, 4, 1, 2) \\
   &(4, 1, 3, -2), \ (4, -1, 3, -2), \ (4, 2, 3, 1), \ (4, 2, 3, -1), \ (-4, 2, 3, 1).
\end{align*}
All length 4 patterns come from $A_{3}$ root subsystems.   
\end{ex}

\begin{ex}
  In type $D_{4}$, there are 49 singular Schubert varieties, and the
  only element which does not comes from $A_{3}$ root subsystems is
  $w= s_{2} \cdot s_{1}s_{3}s_{4} \cdot s_{2} = \bar{1} 4 \bar{3}2$.
  Thus, for all simply laced types, there are only 3 bad patterns to
  consider: $3412, 4231$, and $\bar{1} 4 \bar{3}2$.  It is instructive
  to look at the singular locus of the Schubert varieties for each of
  these 3 patterns:
\begin{align*}
&\mathrm{Sing} \ X_{s_{2} s_{1}s_{3}  s_{2}} = X_{s_{2}}  &(3412 \text{ case}), \\
&\mathrm{Sing} \ X_{s_{3} s_{1}s_{2}  s_{1} s_{3}} = X_{s_{1} s_{3}} &(4231 \text{ case}), \\
&\mathrm{Sing} \ X_{s_{2} s_{1}s_{3}s_{4}  s_{2}} = X_{s_{2}} & (\bar{1} 4 \bar{3}2 \text{ case}).
\end{align*}
\end{ex}

\bigskip

As we mentioned in Section~\ref{s:properties}, the definition of a
Kazhdan-Lusztig polynomial $P_{v,w}(t)$ easily generalizes to all
Coxeter groups.  We use these polynomials to define the notion of a
rationally smooth Schubert variety.  This avoids the more general
definition in terms of \'etale cohomology.  

\begin{defn}
  A point $v \in X_w$ is \textit{rationally smooth} if and only if
  $P_{v,w}(t)=1$.  A Schubert variety $X_w$ is \textit{rationally
    smooth} if every point of $X_w$ is rationally smooth.
\end{defn}


The following theorem as stated is due to Carrell and Peterson.
Related results also appear in Jantzen's book \cite[Ch.5]{jantzen} in
slightly different language.

\begin{thm}
 \cite{carrell94}
The following are equivalent.
\begin{enumerate}
\item $X_{w}$ is rationally smooth at $v$.
\item $P_{v,w}(t)=1$
\item The Bruhat graph on $[v,w]$ is regular of degree $l(w)-l(v)$.  
\end{enumerate}
\end{thm}

In the next theorem, the third condition is due to Carrell-Peterson
\cite{carrell94}.  The fourth condition combines work of
Garsharov~\cite{gasharov97} in type $A$, \cite{B3} for types $B$ and
$C$, then it was conjectured to hold for all Weyl groups by McGovern
and proved by Akyildiz-Carrell \cite{AkyilizCarrell} for types $D$ and
$E$.  It can be checked by computer for $F_4$ and $G_2$ can be done
easily.  The next condition is due to Oh-Yoo \cite{oh.yoo.2010}.  The
last condition is due to Slofstra \cite{slofstra.2013}.
\begin{thm}
The following are equivalent for all Weyl groups.  
\begin{enumerate}
\item $X_{w}$ is rationally smooth.
\item $P_{id,w}(t)=1$
\item $P_{w}(t) = \sum_{v\leq w} t^{\ell(v)}$ is palindromic.  
\item $P_{w}(t) = \prod (1+t+t^{2}+\dots +t^{e_{i}})$ 
\item The Poincar\'{e} polynomial $P_w(t)$ is equal to the generating
  function $R_w(t)$ for the number of regions $r$ in the complement of
  the inversion hyperplane arrangement for $w$ weighted by the
  distance of each region to the fundamental region.
\item The inversion arrangement for $w$ is free and the number of
  chambers of the arrangement is equal to the size of the Bruhat
  interval $[id,w]$.
\end{enumerate}
\end{thm}

\bigskip

For all finite Weyl groups, rational smoothness can be characterized
by pattern avoidance.

\begin{thm}\label{BP rat sm for Schubert}
(Billey-Postnikov \cite{BP-smooth})
$X_{w}$ is rationally smooth if and only for every \textit{stellar} parabolic subgroup $W^{U}$,
$X_v$ for $v=\fl_U(w)$ is rationally smooth in $G^{U}/B^{U}$.
\end{thm}

Note that there are only 2 patterns in $A_{3}$, 6 patterns of type
$B_{3}$ and $C_{3}$, 1 pattern of type $D_{4}$ which should be avoided
by $w$ in order for $X_w$ to be rationally smooth.  The Coxeter
pattern map made a very large reduction in the number of patterns one
needs to remember for both smoothness and rational smoothness.

\bigskip

\begin{rem}
  Smoothness implies rational smoothness.  In terms of the patterns
  characterization, the difference between smoothness and rational
  smoothness for all Weyl group types is just 6 additional patterns, 1
  pattern in $B_{2}$ and 5 patterns of type $G_{2}$.
\end{rem}


\noindent
\textit{Outline of proof of Theorems~\ref{BP sm for Schubert} and~\ref{BP rat sm for Schubert}}

\begin{itemize}
\item Step 1: For classical types $B,C,D$, use Lakshmibai's
  characterization of the tangent space basis to get the general
  smoothness results.  
\item Step 2: Use an analog of Gasharov's theorem to the factor of
  Poincar\'e polynomial for any signed permutation not containing a
  singular pattern to get the rational smoothness of $B,C,D$ which
  extends to all finite types.
\item Step 3: Use Kumar's criterion for (rational) smoothness in the
nil-Hecke ring to test $G_{2}$ and $F_{4}$ by computer.   
\item Step 4: Run a massive parallel computation on the 696,729,600
elements $w \in E_{8}$. 
\begin{itemize}
\item If $w$ has a pattern from type $A$ or $D$, calculate the coefficient of $t^{1}$ and $t^{\ell(w)-1}$ and compare, if different, $w$ is done.  If not, calculate the coefficient of $t^{2}$
and $t^{\ell(w)-2}$, etc.  Eventually one pair differed in every case.
\item If $w$ avoids all patterns from type $A$ or $D$, use analog of Gasharov's algorithm for factoring $P_{w}(t)$.
\end{itemize}
\qed
\end{itemize}

\bigskip

Note that smoothness of $X_w$ automatically implies rational
smoothness. Deodhar proved the following property for type $A$, and
later Peterson proved that it also holds for type $D$ and $E$
(unpublished).  See \cite{CK-2003} for a proof.  A proof for all
finite Weyl group types except $E_6 , E_7, E_8$ follows easily from
Theorem~\ref{BP sm for Schubert} and Theorem~\ref{BP rat sm for
  Schubert}.  For $E_6 , E_7, E_8$, the Peterson theorem is used in
the proof of these two theorems.

\begin{thm}\label{thm:ade}
\emph{(}Deodhar, Peterson, Carrel-Kuttler\emph{)} 
For types $A,D,E$, a Schubert variety $X_{w}$ is smooth if and only if it is rationally smooth.
\end{thm}

A new proof of Theorem~\ref{thm:ade} has recently been announced by
Richmond and Slofstra \cite{richmond.slofstra.2014}.  In fact, they
show that every rationally smooth Schubert variety in any finite Lie
type is an iterated fibre bundle of Grassmannians.  This generalizes
the work in type $A$ by Ryan \cite{ryan}, Wolper \cite{Wolper}, and
Gasharov-Reiner \cite{GR2000} mentioned in Pattern Property 1.

Note that smoothness and rational smoothness are not equivalent for
affine type $\widetilde{A}_{n}$ by Mitchell \cite{Mitchell-86} and
Billey-Crites \cite{Billey-Crites}.

\bigskip

The definition of the Coxeter pattern map also has applications to the 
geometry of Schubert varieties for Weyl groups and affine Weyl
groups. Once again, let $U\subset V$ be a linear subspace. We denote
by $M(x,w; U)$ the set of maximal elements in $[id,w] \cap W^{U}x$
with respect to a new partial order $\leq_{x}$ defined by
\[
wx \leq_{x} w'x  \quad \text{if} \quad \fl(wx) \leq^{U} \fl (w'x).
\]

\begin{thm}\label{thm:BB}
\emph{(}Billey-Braden \cite{BiBr}\emph{)}
If $x,w \in W$, then 
\[
P_{x,w}(1)\geq
\sum_{y \in M(x,w;U)} P_{y,w}(1)P^{U}_{\fl(x),\fl(y)}(1).
\]
\end{thm}
 
\bigskip

\begin{cor}\label{cor:BB} \cite{BiBr} For all $x<w$,  
$P_{x,w}(1)\geq P^{U}_{\fl(x),\fl(w)}(1)$.
\end{cor}

Historically, Theorem~\ref{thm:BB} and Corollary~\ref{cor:BB} were the
first application of simultaneous pattern embedding/flattening on two
Coxeter group elements.  For $u,v \in S_m$ and $x,y \in S_n$, if
$[u,v]$ interval pattern embeds into $[x,y]$ using indices $1\leq i_1
< \ldots < i_m\leq n$ then one can construct a subspace $U$ such that
$u=\fl_U(x), v=\fl_U(y)$ by considering all the roots indexed by
values in the set $\{x_{i_1}, \ldots, x_{i_m}\}$.  Thus,
Corollary~\ref{cor:BB} implies one direction of the stronger type $A$
result in Theorem~\ref{thm:woo.yong}.  From this point of view,
Theorem~\ref{thm:BB} and Corollary~\ref{cor:BB} were precursors to the
notion of interval pattern avoidance introduced in \cite{WooYong}.

Woo \cite{woo.2010} extended the notion of interval pattern avoidance
to other Weyl groups and proved that many of the nice properties in
\cite{WooYong} continue to hold.  In particular, the analog of
Theorem~\ref{thm:woo.yong} holds for all Weyl groups
\cite[Cor. 3.3]{woo.2010}.  Furthermore, Woo relates interval pattern
embeddings with isomorphism of Richardson varieties which are
intersections of two Schubert varieties with respect to two generic
flags.

\begin{thm} \cite[Thm. 3.1]{woo.2010} Let $W',W$ be Weyl groups.
  Suppose there is some root subsystem embedding which embeds $[u, v]
  \subset W'$ in $[x, w]\subset W$ . Then the Richardson varieties
  $X_v^u$ and $X_w^x$ are isomorphic.
\end{thm}

Corollary~\ref{cor:BB} also gives rise to filtrations on
permutations.  

\begin{cor}
  For each $m$, $KL_{m}=\{w \in S_{\infty } \ | \ P_{id,w}(1)\leq m
  \}$ is closed under taking patterns.
\end{cor}

\bigskip

It is interesting to ask for a geometrical explanation for why
(rational) smoothness of Schubert varieties can be characterized by
Coxeter patterns.  The following theorem proves one direction of this.
The other direction is still open: namely, why are patterns from
stellar Coxeter graphs enough.   

\begin{thm}
\emph{(}Billey-Braden \cite{BiBr}\emph{)}
If $X_{\fl (w)}^{U}$ is singular, then $X_{w}$ is singular.   
\end{thm}

\noindent \textit{Outline of proof.} 
\begin{itemize}
\item Realize $G^{U}/B^{U}$ as the fixed points of a certain torus action.

\item Use a theorem of Fogarty-Norman saying that for all smooth algebraic
$T$-schemes $X$ the fixed point scheme $X^{T}$ is smooth.
\end{itemize}
\qed

\bigskip

Several other nice pattern avoidance properties in Coxeter groups are
also known:
\begin{enumerate}
\item (Stembridge \cite{Stem3}) The fully commutative elements in types $B$ and $D$ are characterized with signed patterns.

\item (R.Green \cite{green4}) The fully commutative elements in the
  affine Weyl group of type $A$ are exactly the 321-avoiding
  elements.  

\item (Reading \cite{Reading}) Coxeter-sortable elements are characterized and it is shown that they are equinumerous with clusters and with noncrossing partitions.

\item (Billey-Jones \cite{Billey-Jones}) Deodhar elements for all Weyl groups are characterized.

\item (Billey-Crites \cite{Billey-Crites}) The rationally smooth
  Schubert varieties in the affine type A flag manifold are
  characterized as 3412, 4231 avoiding plus one extra family of
  twisted spiral varieties.  Crites also studied the enumeration of
  affine permutations indexing  rationally smooth Schubert varieties in
  \cite{Crites1}.  

\item (Chen-Crites-Kuttler, manuscript) A Schubert variety $X_{w}$ of affine type A is smooth if and only if $w \in \widetilde{S}_{n}$ avoids $3412$ and $4231$.  Furthermore, the tangent space to $X_{w}$ at the identity can be described in terms of reflection over real and imaginary roots.

\item (Matthew Dyer, manuscript) Smooth and rationally smooth
  {S}chubert varieties can be detected using rank 2 subvarieties \cite{dyer.rank.2}. 

\item (Matthew Samuel, manuscript) Affine Schubert varieties for all types can be characterized by patterns using a new version of pattern avoidance for Coxeter groups based on reflection groups.
\end{enumerate}

\section{Computer tools for Schubert geometry}
In the lecture series that gave rise to this article, we discussed some
computer tools for the study of geometry of Schubert varieties and for
more general topics in mathematics.  The video file of the lecture
devoted to the contents of this section is available at the following
website.
\[
\textrm{http://mathsoc.jp/en/videos/2012msj-si.html}
\]
The main ideas presented are pertaining to computer proofs, the Online
Encyclopedia of Integer Sequences, the Database of Permutation Pattern
Avoidance and Sage.  The demos in the lecture are best seen online so
we will not include that discussion here.

We do want to highlight one of the Sage demos discussed, because it is
related to some recent developments on marked mesh patterns which
unify the descriptions of several pattern avoidance properties for
permutations using the language of \textit{marked mesh patterns}.

\begin{defn}
  (Br\"anden and Claesson \cite{BrCl}) A \textit{mesh pattern} is a
  permutation matrix with shaded regions between certain entries.
\end{defn}
\vspace{-0.1in}
\[
\unitlength 0.1in
\begin{picture}( 32.0000, 11.8000)( 13.3000,-18.5000)
%
\special{pn 8}%
\special{pa 2200 840}%
\special{pa 3040 840}%
\special{fp}%
%
\special{pn 8}%
\special{pa 2200 980}%
\special{pa 3040 980}%
\special{fp}%
%
\special{pn 8}%
\special{pa 2200 1120}%
\special{pa 3040 1120}%
\special{fp}%
%
\special{pn 8}%
\special{pa 2200 1260}%
\special{pa 3040 1260}%
\special{fp}%
%
\special{pn 8}%
\special{pa 2200 1400}%
\special{pa 3040 1400}%
\special{fp}%
%
\special{pn 8}%
\special{pa 2340 1540}%
\special{pa 2340 700}%
\special{fp}%
%
\special{pn 8}%
\special{pa 2480 1540}%
\special{pa 2480 700}%
\special{fp}%
%
\special{pn 8}%
\special{pa 2620 1540}%
\special{pa 2620 700}%
\special{fp}%
%
\special{pn 8}%
\special{pa 2760 1540}%
\special{pa 2760 700}%
\special{fp}%
%
\special{pn 8}%
\special{pa 2900 1540}%
\special{pa 2900 700}%
\special{fp}%
%
\special{pn 8}%
\special{sh 1.000}%
\special{ar 2340 980 30 30  0.0000000 6.2831853}%
%
\special{pn 8}%
\special{sh 1.000}%
\special{ar 2480 1400 30 30  0.0000000 6.2831853}%
%
\special{pn 8}%
\special{sh 1.000}%
\special{ar 2620 840 30 30  0.0000000 6.2831853}%
%
\special{pn 8}%
\special{sh 1.000}%
\special{ar 2760 1120 30 30  0.0000000 6.2831853}%
%
\special{pn 8}%
\special{sh 1.000}%
\special{ar 2900 1260 30 30  0.0000000 6.2831853}%
%
\special{pn 8}%
\special{pa 2340 980}%
\special{pa 2760 980}%
\special{pa 2760 1120}%
\special{pa 2340 1120}%
\special{pa 2340 980}%
\special{ip}%
%
\special{pn 4}%
\special{pa 2648 980}%
\special{pa 2508 1120}%
\special{fp}%
\special{pa 2606 980}%
\special{pa 2466 1120}%
\special{fp}%
\special{pa 2564 980}%
\special{pa 2424 1120}%
\special{fp}%
\special{pa 2522 980}%
\special{pa 2382 1120}%
\special{fp}%
\special{pa 2480 980}%
\special{pa 2348 1114}%
\special{fp}%
\special{pa 2438 980}%
\special{pa 2340 1078}%
\special{fp}%
\special{pa 2396 980}%
\special{pa 2340 1036}%
\special{fp}%
\special{pa 2690 980}%
\special{pa 2550 1120}%
\special{fp}%
\special{pa 2732 980}%
\special{pa 2592 1120}%
\special{fp}%
\special{pa 2760 994}%
\special{pa 2634 1120}%
\special{fp}%
\special{pa 2760 1036}%
\special{pa 2676 1120}%
\special{fp}%
\special{pa 2760 1078}%
\special{pa 2718 1120}%
\special{fp}%
\put(14.0000,-18.0000){\makebox(0,0)[lb]{The dots represents 1's in the permutation matrix.}}%
%
\special{pn 8}%
\special{pa 1330 670}%
\special{pa 4530 670}%
\special{pa 4530 1850}%
\special{pa 1330 1850}%
\special{pa 1330 670}%
\special{ip}%
\end{picture}%
\]

\begin{defn}
(\'Ulfarsson \cite{Ulfarssonmesh})
A \textit{marked mesh pattern} is a mesh pattern with numbers in the shaded regions.
\end{defn}

The next theorem states that we can also use marked mesh patterns for characterizing Schubert varieties. See \cite{Ulfarssonmesh} for details.

\begin{thm}
\emph{(}\'Ulfarsson \cite{Ulfarssonmesh}\emph{)} 
The smooth, Gorenstein,
factorial, defined by inclusions, and 321-hexagon avoiding
permutations can be described by marked mesh patterns.
\end{thm}

\section{Open Problems}

In addition to the open problems we have mentioned in the text, there
are some more open problems concerning pattern avoidance
properties. We hope that computer experiments will help the reader to
study those problems.

\begin{ques} (Woo-Yong) Characterize the Gorenstein, LCI and factorial locus of $X_w$ using patterns. 
\end{ques}

\begin{ques} (From \'Ulfarsson) Is there a nice generating function to count the number of Gorenstein/LCI permutations or Schubert varieties defined by inclusions, etc.  
\end{ques}

\begin{ques} Find a geometric explanation why a finite number of patterns suffice in all cases above.  
\end{ques}

\begin{ques} What nice properties does the inversion arrangement have for other pattern avoiding families?
\end{ques}

\begin{ques}  $KL_{m}$ is closed under taking patterns by \cite{BiBr}.  Can it always be
  described by a finite set of patterns?  Conjectured to be yes by Billey-Weed-Woo.
\end{ques}

\begin{ques} Conjecture (Woo): The Schubert varieties with
  multiplicity $\leq 2$ can be characterized by pattern avoidance.
  Can this be extended to a pattern avoidance characterization of
  Schubert varieties with multiplicity $\leq k$?  Note, Woo-Yong
  showed that $\mathrm{mult}_x(X_y) = \mathrm{mult}_u(X_v)$ if $[u,v]$
  interval pattern embeds into $[x,y]$ \cite[Cor. 6.15]{WooYong}.
\end{ques}

\begin{ques} What other filtrations on the set of all permutations can be characterized by (generalized) patterns?
\end{ques}

\begin{ques} Describe the maximal singular locus of a Schubert variety
  for other semisimple Lie groups using Coxeter patterns.   
\end{ques}

\begin{ques} Give a pattern based algorithm to produce the factorial and/or
Gorenstein locus of a Schubert variety in other types.
\end{ques}

\begin{ques} Is there a nice generating function to count the number of
smooth, factorial and/or Gorenstein permutations in other types?
\end{ques}


\begin{ques} What is the right notion of patterns for GKM spaces?
\end{ques}

\begin{ques} Say $X_{w}$ is \textit{combinatorially smooth} if
  $\ell(w) = \#\{t_{ij}\colon t_{ij}\leq w\}$.  In $S_n$
  combinatorially smooth is equivalent to smooth by the
  Lakshmibai-Seshadri Theorem.  However, for other Weyl groups this is
  a weaker notion than rational smoothness.    Characterize the
  combinatorially smooth Schubert varieties by generalized pattern
  avoidance.
\end{ques}

\begin{ques} Can Lakshmibai's characterization of the tangent space basis for
$B$, $C$ and $D$ be translated into signed patterns or a signed variation on marked mesh patterns.
\end{ques}

\begin{ques}  What is the analog of marked mesh patterns for other types?   
\end{ques}

\begin{ques} What is the M\"obius function for the poset of pattern
containment on $S_{\infty}$?  See the excellent survey by Einar
Steingrimmson \cite[Sect. 5]{steingrimsson.2013} for more details on
this and other pattern related problems.
\end{ques}

\begin{ques}
  Which of the many pattern avoidance related theorems on Schubert
  varieties have analogs for other interesting families of varieties
  such as the $GL_p\times GL_q$-orbit closures of the flag manifold or
  the Peterson varieties as mentioned at the end of
  Section~\ref{s:properties}?  See
  \cite{Brion99,Hultman2010,Insko.Yong,mcgovern.2009,mcgovern.trapa} for further details.
\end{ques}

\begin{ques} More generally, what other types of theorems have
  canonical representations which might lead to more computer database
  tools? For example, how about hypergeometric series, integer
  sequences, patterns?  See \cite{billey-tenner} for more discussion
  on this topic.
\end{ques}

\subsection*{Acknowledgments}
We offer our immense gratitude to those who helped in preparation of
the lecture series and this article including Michael Albert, Andrew
Crites, Jack Lee, Monty McGovern, Julia Pevtsova, Brendan Pawlowski,
Ed Richmond, Sudeshna Sen, William Slofstra, Joshua Swanson, Henning
\'Ulfarsson, Alex Woo, Alex Yong, and all of the organizers of the MSJ
summer program Megumi Harada, Takeshi Ikeda, Shizuo Kaji, Toshiaki
Maeno, Mikiya Masuda, Hiroshi Naruse, Toru Ohmoto, Norio Iwase.  We
also want to thank the students in Osaka City University who helped
organize the conference, the audience who attended the lectures at the
conference and the readers of this article. Kokoroyori Kannsha
Itashimasu.  Finally, we greatly appreciate the corrections and
comments of an anonymous referee.




\def\cprime{$'$}

\end{document}